\documentclass[12pt]{article}
\usepackage[top=1in, bottom=1in, left=1in, right=1in]{geometry}

\usepackage{xcolor}
\usepackage{geometry}
\usepackage{amssymb,bm}
\usepackage{amsmath}
\usepackage{amsthm}
\usepackage{graphicx}
\usepackage{hyperref}
\usepackage{caption}
\usepackage{float} 
\hypersetup{pdfborder=0 0 0}
\usepackage{calrsfs}
\usepackage{booktabs}
\usepackage{physics}

\renewcommand{\theequation}{{\thesection}.{\arabic{equation}}}

\newtheorem{theorem}{{\sc Theorem}}[section]

\newtheorem{lemma}[theorem]{{\sc Lemma}}
\newtheorem{corollary}[theorem]{Corollary}
\newtheorem{remark}[theorem]{Remark}

\newcommand\restr[2]{{ \left.\kern-\nulldelimiterspace #1 \vphantom{\big|} \right|_{#2}}}
\newcommand{\RR}{\mathbb{R}}
\newcommand{\ZZ}{\mathbb{Z}}
\newcommand{\CC}{\mathbb{C}}

\newcommand{\n}{\noindent}
\newcommand{\comment}[1]{}

\DeclareMathAlphabet{\pazocal}{OMS}{zplm}{m}{n}

\newcommand{\PC}{\pazocal{C}}
\newcommand{\CT}{\pazocal{T}}

\newcommand{\CS}{\pazocal{S}}
\newcommand{\CR}{\pazocal{R}}
\newcommand{\CL}{\pazocal{L}}

\usepackage{bbm}
\usepackage{subcaption}

\newcommand{\ce}{\operatorname{ce}}
\newcommand{\se}{\operatorname{se}}
\newcommand{\Ce}{\operatorname{Ce}}
\newcommand{\Se}{\operatorname{Se}}

\renewcommand{\Res}{\operatorname*{Res}}

\title{Higher-order necessary conditions for nonscattering}
\author{
Narek Hovsepyan\footnote{School of Mathematical and Statistical Sciences, University of Texas Rio Grande Valley, Edinburg, TX, USA (narek.hovsepyan@utrgv.edu)}}
\date{}

\begin{document}
\maketitle

\begin{abstract}
We study the scattering of incident waves by a class of planar analytic inhomogeneities with constant refractive index, governed by the Helmholtz equation. We derive an infinite hierarchy of necessary conditions for nonscattering: any complex analytic nonscattering incident wave must satisfy infinitely many differential constraints at a specific point determined by the geometry of the inhomogeneity. As an application, we show that for any noncircular elliptical inhomogeneity, every single separated (Mathieu) mode necessarily scatters. This is in stark contrast to the circular case, where single separated (Bessel) modes are nonscattering at infinitely many wavenumbers.
\end{abstract}

\section{Introduction} \label{SECT Prelim}
\setcounter{equation}{0}

Inverse scattering seeks to recover geometric or material information about an inhomogeneity $D$ from the field it scatters. A basic question is therefore whether a given incident wave $v$ at a fixed wavenumber $k$ necessarily produces a nontrivial scattered field, or whether nonscattering can occur. A wavenumber $k$ is called nonscattering if there exists a nontrivial incident wave $v$ whose corresponding scattered field vanishes identically. At such a wavenumber, the relative scattering operator is not injective \cite{CK19}, which may lead to the failure of certain reconstruction methods. The nonscattering question is therefore fundamental both for uniqueness in inverse problems and for wave-based imaging \cite{CK19, CCH23}. Nonscattering is closely related to the transmission eigenvalue problem. If $k$ is a nonscattering wavenumber, then the restriction of the incident wave $v$ to $D$ and the corresponding total field satisfy their respective Helmholtz equations in $D$ and share the same Cauchy data on $\partial D$. If we now regard $v$ as an unknown defined only in $D$, rather than the restriction of an incident wave, we obtain the transmission eigenvalue problem. Hence, the set of nonscattering wavenumbers is a (possibly empty) subset of the real transmission eigenvalues. Although infinitely many real transmission eigenvalues are known to exist under broad assumptions \cite{CGH, CCH23}, being a transmission eigenvalue is far from sufficient for nonscattering: the $v$-component of a transmission eigenfunction is defined only in $D$, whereas nonscattering requires it to extend to a global solution of the background Helmholtz equation. We refer to \cite{CV26} for a recent survey on transmission eigenvalues and nonscattering.

Boundary geometry places strong constraints on nonscattering. In particular, for singular inhomogeneities the picture is by now relatively well understood: broadly speaking, singularities lead to scattering. This was established through a sequence of corner-scattering results \cite{BPS,EH15,PSV17, B18, EH,LHY18}, which also led to single-measurement uniqueness results showing that a convex polygonal inhomogeneity can be uniquely determined from the scattering data generated by a single incident wave at a fixed frequency \cite{EH15, HuSaloVesalainen2016} (see also \cite{blas}). More generally, free-boundary methods were used in \cite{CV} to show that for Lipschitz inhomogeneities nonscattering generically forces additional boundary regularity (the initial Lipschitz assumption was relaxed in \cite{SS21} to allow for certain cusps). In particular, for constant refractive index, if the boundary $\partial D$ is not real analytic at a point $\bm{x}_0$, then any incident wave satisfying $v(\bm{x}_0)\neq0$ must scatter. More recently, \cite{SS25} removed the nonvanishing condition $v(\bm{x}_0)\neq0$ under stronger a priori assumptions on $D$. For related anisotropic results we refer to \cite{CX21, CVX23, KSS24, KSS26} and the references therein.

Much less is understood about the scattering properties of real analytic inhomogeneities. For the disk, it is well known that there exists a countably infinite set of nonscattering wavenumbers with associated Herglotz incident waves \cite{CK19,VX} (see also Section~\ref{SECT app ellipse}). A natural question is therefore whether a bounded simply connected domain with real analytic boundary that is not a disk, and with constant refractive index, can admit a nontrivial incident wave that does not scatter. This question has only recently begun to be studied \cite{HV1,HV2} and the present work continues this investigation. We also mention \cite{VX, VX2, X26}, which establish for various classes of smooth inhomogeneities, finiteness of the set of nonscattering wavenumbers associated with a fixed Herglotz incident wave.

This work is a continuation of our work \cite{HV1}, where we studied planar analytic domains whose boundary parametrization
$g(t)=x_1(t)+ix_2(t)$, $t\in[-\pi,\pi]$, extends to a complex analytic function of $t$ in a domain $\CT\subset\CC$ and has a simple critical point $t_0\in\CT$. Under a suitable topological assumption--namely, that there exists a contour joining the endpoints $\pm\pi$, passing through $t_0$ and contained in $\CT$, along which $\Re g$ attains its maximum at $t_0$--we used the method of steepest descent to derive two necessary conditions for the nonscattering of a general incident wave admitting a complex analytic extension. The first-order condition, which was the main focus of that work, has a particularly simple form: it requires the incident wave $v$ to vanish at the point $\bm{x}_0=\left(x_1(t_0),x_2(t_0)\right)\in\CC^2$.
The second-order condition, however, is considerably more complicated and is given in \cite{HV1} by an involved expression in terms of $v$. In the present work, we derive an infinite hierarchy of necessary conditions and uncover a fairly simple (reduced) structure: namely, any complex analytic nonscattering incident wave $v$ must satisfy

\begin{equation} \label{L_Nv = 0}
L_N v (\bm{x}_0) = 0 \qquad \text{for all} \ N=0,1,...,
\end{equation}

\n where $L_N$ is a linear partial differential operator in two variables of order $2N$ with a particular structure (cf. Theorem~\ref{THM reduced}), whose coefficients depend on the domain $D$, its constant refractive index and the wavenumber $k$. We also provide these higher-order necessary conditions in unreduced form, as a closed-form expression involving derivatives of an explicit function. In fact, these are the formulas used to derive the reduced conditions \eqref{L_Nv = 0}. Moreover, they can be used to determine the operator $L_N$ explicitly for any $N$. In particular, $L_0=\operatorname{Id}$ and the second-order condition, which involves $L_1$ takes a simpler form, allowing us to apply it in new settings not considered in \cite{HV1}. Specifically, when $D$ is a truly elliptical inhomogeneity (i.e., not a disk), we show that an incident wave consisting of a single separated mode (given in terms of Mathieu functions) always scatters. This is in stark contrast to the disk case, where an incident wave consisting of a single separated mode (given in terms of Bessel functions) is known to be nonscattering at infinitely many distinct positive wavenumbers. We also compute the third-order condition involving the operator $L_2$ (we only state the result of this computation, which was carried out using Maple and is available in \cite{maple}). We remark that the operators $L_0$ and $L_1$ are independent of the refractive index and the wavenumber $k$. These parameters first appear in the operator $L_2$. We further observe that a critical point of $h(t)=x_1(t)-ix_2(t)$ plays an equally important role. Under an analogous topological assumption, such a critical point gives rise to its own infinite hierarchy of necessary conditions for nonscattering, imposed at the corresponding point. In particular, for a truly elliptical inhomogeneity this yields a second hierarchy analogous to \eqref{L_Nv = 0}, imposed at a point $\tilde{\bm{x}}_0$ symmetric to $\bm{x}_0$.  

The infinite hierarchy \eqref{L_Nv = 0} further illustrates the highly constrained and delicate nature of nonscattering in nonradial geometries and raises the question of whether such nonscattering is possible at all. This is a natural question in inverse scattering and is of additional interest in view of the known analogy between the nonscattering problem and the Schiffer, or Pompeiu, problem (see Remark~\ref{REM Schiff}).

\section{Preliminaries} \label{SECT Prelim}
\setcounter{equation}{0}

\n Consider a nontrivial incident wave $v$ at a fixed wavenumber $k>0$, satisfying

\begin{equation} \label{v Helm}
\Delta v + k^2 v = 0 \qquad \qquad \text{in} \ \RR^2.
\end{equation}

\n Let $D$ be a bounded, simply connected, smooth planar domain with constant refractive index $0<q_0\neq 1$. Let $u^{\text{tot}}$ denote the total field in the presence of the inhomogeneity $D$ and the incident wave $v$ and set $u=u^{\text{tot}}-v$. We say that $v$ is nonscattering for $(D,q_0)$, or simply nonscattering (since $D$ and $q_0$ will be fixed and clear from the context), if and only if $u=0$ outside $D$. Equivalently, this means that the following problem admits a solution in $D$:

\begin{equation} \label{u system}
\begin{cases}
\Delta u + k^2 q_0 u = k^2(1-q_0)v
& \text{in} \ D,
\\
u = \partial_{\bm{n}}u = 0
& \text{on} \ \Gamma=\partial D,
\end{cases}
\end{equation}

\n where $\bm{n}$ denotes the unit outer normal to $\Gamma$. Throughout, we use boldface notation for vector quantities. It will be convenient to introduce the wavenumber in the medium $D$ by

\begin{equation} \label{kappa}
\kappa = k\sqrt{q_0}.
\end{equation}

\n Multiplying the PDE in \eqref{u system} by a test function $\phi$ and using Green's identity, we obtain

\begin{equation*} 
\int_D v \phi d\bm{x} = 0 \qquad \qquad \text{for any} \quad \Delta \phi + \kappa^2 \phi = 0.
\end{equation*}

\n We now take as a test function a plane wave with a complex wave vector. Namely, let $\phi(\bm{x})=e^{i\bm{x}\cdot\bm{\zeta}}$, where $\cdot$ denotes the inner product of $\RR^2$ and for $\phi$ to satisfy the required PDE, we must have $\bm{\zeta}\cdot\bm{\zeta}=\kappa^2$. Thus, we obtain the following necessary condition for $v$ to be nonscattering:

\begin{equation} \label{v exp 0}
\int_D v e^{i \bm{x} \cdot \bm{\zeta}} d\bm{x} = 0 \qquad \qquad \forall \, \bm{\zeta} \in \CS_\kappa = \left\{ \bm{\zeta} \in \CC^2 : \bm{\zeta} \cdot \bm{\zeta} = \kappa^2 \right\}.
\end{equation}

\begin{remark} \label{REM Schiff}
\normalfont

Let us comment on the analogy between the above formulation and the failure of the Pompeiu (or equivalently, Schiffer) property. We refer to \cite{zalc1} for a survey of the Pompeiu problem, see also the appendix of \cite{HV1}. Here we recall an equivalent characterization of the failure of the Pompeiu property: $D$ fails to have the Pompeiu property iff there exists $\kappa>0$ such that

\begin{equation} \label{pomp}
\int_D e^{i \bm{x} \cdot \bm{\zeta}} d\bm{x} = 0 \qquad \qquad \forall \, \bm{\zeta} \in \CS_\kappa.
\end{equation}

\n The analogy is now evident: if in the nonscattering problem we formally replace $v$ by a constant, then we recover the failure of the Pompeiu property with $\kappa=k\sqrt{q_0}$. Of course, $v$ cannot be constant, since it must satisfy the Helmholtz equation with wavenumber $k>0$. It is well known that the disk fails to have the Pompeiu property and the celebrated Pompeiu conjecture stated that, modulo sets of measure zero, every other simply connected planar domain has the Pompeiu property. This conjecture was recently disproved in \cite{CS, conject}.
\end{remark}

To obtain necessary conditions for nonscattering from the integral condition \eqref{v exp 0}, we choose a parametrized family of test vectors $\bm{\zeta}\in\CS_\kappa$. Under appropriate assumptions on $D$ and an analyticity assumption on $v$ (cf. Section~\ref{SECT main results}), we derive a full asymptotic expansion of the resulting integral as the parameter tends to infinity. If $v$ is nonscattering, then all the coefficients in this expansion must vanish. These are precisely the higher-order necessary conditions for nonscattering, which can be reduced to the form \eqref{L_Nv = 0} mentioned in the introduction. Conversely, if any one of these infinitely many coefficients is nonzero, then $v$ must scatter. To this end, let us choose vectors from $\CS_\kappa$ parametrized by $\lambda > 0$ and defined by

\begin{equation} \label{xi lambda}
\bm{\zeta}(\lambda) = (- i \lambda, \sqrt{\lambda^2 + \kappa^2}).
\end{equation}

\n We then reduce the integral in \eqref{v exp 0} to an integral over the boundary $\Gamma=\partial D$, using the equations satisfied by $v$ and $\phi$ (cf. \cite{HV1} for details). In summary, if $v$ is nonscattering, then

\begin{equation} \label{I cali}
I(\lambda) := \int_{\Gamma} \bm{n} \cdot \left( \nabla v - i \bm{\zeta}(\lambda) v \right) e^{i \bm{x} \cdot \bm{\zeta}(\lambda)} ds = 0 \qquad \qquad \forall \, \lambda > 0.
\end{equation}

Let us now summarize and discuss the structure of the paper. Our goal is to obtain the full asymptotic expansion of $I(\lambda)$ as $\lambda\to+\infty$. This will be achieved using the method of steepest descent, and our assumptions on $D$ and $v$ are precisely those needed to ensure that this method is applicable. The integral $I$ is not in a form to which the steepest descent method is readily applicable. As already discussed, in \cite{HV1} we singled out its dominant parts and obtained the leading two terms in the expansion. In this work, we study this integral more closely. Roughly speaking, we rewrite it as an infinite sum of integrals, to each of which the steepest descent method is readily applicable and yields a full asymptotic expansion. Combining all these expansions gives the desired full expansion of $I$ (cf. Section~\ref{SECT unreduced}). We obtain the coefficients of this expansion in closed form, expressed solely in terms of derivatives of an explicitly defined auxiliary function $G$, whose dependence on the incident wave $v$ is also explicit (cf. Theorem~\ref{THM 1} and Remark~\ref{REM residue}). Consequently, the differentiations can be carried out to obtain the coefficient at any desired order explicitly. Nonscattering requires all of these coefficients to vanish, and we refer to the resulting conditions as the unreduced form of the nonscattering conditions.

We next analyze these conditions and study their structure. In Section~\ref{SECT reduced} we show that, using the Helmholtz equation satisfied by $v$, they can be reduced to the PDE form \eqref{L_Nv = 0}. Namely, $L_N v(\bm{x}_0)=0$ for any $N=0,1,...$, where, recall, $\bm{x}_0$ is associated with a critical point of the boundary parametrization $g(t)=x_1(t)+ix_2(t)$. We refer to these conditions as the reduced form of the nonscattering conditions. We determine the general structure of the linear differential operators $L_N$, although we do not obtain explicit formulas for their coefficients, except for the two highest-order coefficients. The operator $L_N$ can nevertheless be computed fully explicitly from the unreduced form of the nonscattering conditions. Carrying out these computations, we obtain explicit formulas for $L_1$ and $L_2$, while $L_0=\operatorname{Id}$. For convenience, in stating our results we reverse this order and begin with the reduced form in Section~\ref{SECT reduced}, followed by the unreduced form in Section~\ref{SECT unreduced}. We also observe that choosing the other branch of the square root in the test vector $\bm{\zeta}(\lambda)$ leads instead to a critical point of the function $h(t)=x_1(t)-ix_2(t)$, see Remarks~\ref{REM h crit point} and \ref{REM h crit point 2}. Finally, in Section~\ref{SECT app ellipse}, we apply our results to a truly elliptical inhomogeneity. Using the first- and second-order conditions, we show that incident waves consisting of a single separated mode always scatter.

\section{Main results} \label{SECT main results}
\setcounter{equation}{0}

Let us describe our assumptions and introduce some notation:

\begin{enumerate}

\item[(H1)] $\Gamma$ is a simple closed curve parametrized by $\bm{x}(t) = (x_1(t), x_2(t))$ for $t \in [-\pi, \pi]$, where $x_j(t)$ for $j=1,2$ extend to complex analytic functions in an open set $\CT \subset \CC$, which contains the interval $[-\pi, \pi]$. We also set

\begin{equation} \label{g and h}
g(t) = x_1(t) + ix_2(t), \qquad \qquad h(t) = x_1(t) - ix_2(t).
\end{equation}

\item[(H2)] $D$, the domain enclosed by $\Gamma$, has constant refractive index $0<q_0 \neq 1$. 

\item[(H3)] $g(t)$ has a simple critical point at $t_0 \in \CT$, i.e., $g'$ has a simple zero at $t_0$, and $t_0 \neq \pm \pi$. Further, $h'(t_0) \neq 0$. We also set $\bm{x}_0 = \bm{x}(t_0)$.

\item[(H4)] There exists a contour $\PC\subset\CT$ joining $-\pi$ to $\pi$ and passing through $t_0$, such that $\Re g(t) < \Re g(t_0)$ for every $t \in \PC \backslash \{t_0\}$.

\item[(H5)] $v(\bm{x})$ solves the Helmholtz equation \eqref{v Helm} and extends to an analytic function of two complex variables in some open subset of $\CC^2$ containing $\left\{\bm{x}(t): t \in \CT \right\}$.
 
\end{enumerate}

\begin{remark} \mbox{}
\normalfont
\begin{enumerate}
\item[$\bullet$] The requirement $h'(t_0)\neq 0$ in (H3) is a nondegeneracy assumption on the parametrization $\bm{x}(t)$. Indeed, in view of the assumption $g'(t_0)=0$, it is equivalent to $\bm{x}'(t_0)\neq 0$.

\item[$\bullet$] A particularly important class of incident waves is given by the Herglotz waves -- continuous superpositions of plane waves of the form

\begin{equation*}
v(\bm{x}) = \int_{\mathbb{S}^1} \psi(\bm{\eta}) e^{ik\bm{x} \cdot \bm{\eta}} ds(\bm{\eta}),
\end{equation*}

\n where $\mathbb{S}^1$ denotes the unit circle centered at the origin and $\psi\in L^2(\mathbb{S}^1)$ is called the density function. For such incident waves assumption (H5) is automatically satisfied, since $v(\bm{x})$ extends to an entire function of $\bm{x}\in\CC^2$. One may also consider finite linear combinations of plane waves. We note that a single plane wave always scatters for the inhomogeneities considered here, as was already established in \cite{HV1}. Indeed, a single plane wave is an exponential function and therefore has no zeros. Thus, it cannot satisfy the first-order necessary condition $v(\bm{x}_0)=0$.
\end{enumerate}
\end{remark}

\n We also introduce the following concise notation: for $n=0,1,...$,

\begin{equation} \label{g_n h_n}
g_n = g^{(n)}(t_0), \qquad \qquad h_n = h^{(n)}(t_0) . 
\end{equation}

\n Finally, we introduce the differential operators 

\begin{equation} \label{A B}
A = \partial_{x_1} - i \partial_{x_2}, \qquad \qquad B = \partial_{x_1} + i \partial_{x_2}
\end{equation} 

\n and note that the Helmholtz equation \eqref{v Helm} takes the form

\begin{equation} \label{Helmholtz AB}
AB v = BA v = -k^2 v.
\end{equation}

\subsection{Reduced form} \label{SECT reduced}

\begin{theorem} \label{THM reduced}
Assume (H1)--(H5). There exists a family of linear differential operators $\{L_N\}_{N=0}^\infty$ in the variables $x_1,x_2$ such that, if the incident wave $v$ is nonscattering, then

\begin{equation*}
L_N v(\bm{x}_0)=0, \qquad \qquad \forall \, N=0,1, ... .
\end{equation*}

\n Each $L_N$ is of order $2N$ and has the form

\begin{equation} \label{L_N form}
L_N = a_{N} A^N + \sum_{n=1}^{2N}b_{N,n}B^n,
\end{equation} 

\n where $a_N$ and $b_{N,n}$ are complex constants. Moreover, the highest-order coefficients are always nonzero and are given explicitly by 

\begin{equation} \label{a_N, b_N,2N}
a_N = \frac{(2N+1)!}{N!} g_2^N, \qquad \qquad b_{N, 2N} = h_1^{2N}.
\end{equation}
\end{theorem}

The proof of this result is given in Section~\ref{SECT proof of reduction}. Let us note that $L_0=\operatorname{Id}$, and in Section~\ref{SECT R_1} we compute the operator $L_1$. Consequently, the first two conditions take the form

\begin{equation} \label{1st and end cond v}
\begin{cases}
v(\bm{x}_0) = 0,
\\[.1in]
6g_2 A v(\bm{x}_0) + \left( 6 h_2 - \frac{2g_3 h_1}{g_2} \right)B v(\bm{x}_0) + h_1^2 B^2 v(\bm{x}_0) = 0.
\end{cases}
\end{equation}

\begin{remark} \label{REM h crit point}
\normalfont Instead of considering critical points of $g=x_1+ix_2$, one may also consider critical points of $h=x_1-ix_2$ and obtain analogous necessary conditions for nonscattering. In other words, the roles of $g$ and $h$ can be interchanged. We refer to Remark~\ref{REM h crit point 2} below for a further explanation of why this is possible. In particular, suppose that $h$ has a simple critical point at $\tilde{t}_0\in\CT$ and that the analogue of the contour assumption (H4) holds. Let $\tilde{\bm{x}}_0=\bm{x}(\tilde{t}_0)$. If $v$ is nonscattering, then

\begin{equation*}
\tilde{L}_N v (\tilde{\bm{x}}_0) = 0, \qquad \qquad \forall \, N=0, 1, ...
\end{equation*} 

\n Here $\tilde{L}_N$ has the same form as $L_N$ in \eqref{L_N form}, with the roles of $A$ and $B$ interchanged. In particular, the first two conditions read

\begin{equation} \label{1st and 2nd cond with h}
\begin{cases}
v(\tilde{\bm{x}}_0) = 0
\\[.1in]
6 \tilde{h}_2 B v(\tilde{\bm{x}}_0) + \left( 6 \tilde{g}_2 - \frac{2\tilde{h}_3 \tilde{g}_1}{\tilde{h}_2} \right)A v(\tilde{\bm{x}}_0) + \tilde{g}_1^2 A^2 v(\tilde{\bm{x}}_0) = 0,
\end{cases}
\end{equation}

\n where $\tilde{g}_n = g^{(n)}(\tilde{t}_0)$ and similarly $\tilde{h}_n$ is defined.
\end{remark}

\begin{remark} \label{REM 3rd order}
\normalfont

The first two conditions \eqref{1st and end cond v} do not involve the refractive index $q_0$ or the wavenumber $k$. Both of these parameters appear in the third-order condition $L_2 v(\bm{x}_0) = 0$, which takes the following form (the computation was carried out using Maple, see \cite{maple}):

\begin{equation*}
L_2 = 60g_2^2 A^2 + b_{2,1} B + b_{2,2} B^2 + b_{2,3} B^3 + h_1^4 B^4,
\end{equation*}

\n where

\begin{equation*}
\begin{split}
b_{2,1} = 4k^2(4q_0-5)g_2h_1^2 + 40h_4 -\frac{80g_3}{g_2}h_3
+ 40 \left(\frac{2g_3^2}{g_2^2}-\frac{g_4}{g_2}\right)h_2
+8\left(\frac{5g_3g_4}{g_2^2}-\frac{g_5}{g_2}-\frac{40g_3^3}{9g_2^3}
\right)h_1
\end{split}
\end{equation*}

\n and

\begin{equation*}
\begin{split}
b_{2,2} = & \frac{40}{3} \left(2 h_3 -\frac{5g_3}{g_2}h_2 \right) h_1 + 60h_2^2
+\frac{20}{9} \left(\frac{8g_3^2}{g_2^2}-\frac{3g_4}{g_2}\right)h_1^2,
\\[.05in]
b_{2,3} = & 20 \left(h_2-\frac{g_3}{3g_2} h_1\right) h_1^2.
\end{split}
\end{equation*}

\end{remark}

\subsection{Unreduced form} \label{SECT unreduced}

One of the key steps in the analysis of the integral $I$ in \eqref{I cali} is to rewrite it in a form suitable for the application of the steepest descent method and the derivation of its full asymptotic expansion. To this end, we first define

\begin{equation} \label{rho}
\rho(\epsilon) = \frac{\sqrt{1+\epsilon^2 \kappa^2} - 1}{\epsilon}.
\end{equation}

\n Note that $\rho$ is analytic in a neighborhood of $\epsilon=0$ once we set $\rho(0)=0$. As was already observed in \cite{HV1}, two orders of cancellation occur in the integral $I$ that are not apparent from its original form \eqref{I cali}. In the Appendix, we make these cancellations explicit by integrating by parts twice and, after suitable simplifications, show that $I$ can equivalently be written in the following form:

\begin{equation} \label{I cali G g}
I(\lambda) =  \frac{i}{\lambda} \int_{-\pi}^\pi G(t, \lambda^{-1}) e^{\lambda g(t) }dt,
\end{equation}

\n where 

\begin{equation} \label{G}
G(t, \epsilon) = \frac{d}{dt} \left[ e^{i \rho(\epsilon) x_2(t)} Av(\bm{x}(t)) \right] + \frac{\kappa^2 - \rho(\epsilon)^2}{2} h'(t) e^{i \rho(\epsilon) x_2(t)} v(\bm{x}(t)). 
\end{equation}

\n Note that, as $\lambda\to+\infty$ we have $\epsilon=1/\lambda\to0$ and the key advantage of this formulation is that $G$ is analytic in a neighborhood of $\epsilon=0$. Since $t_0$ is a simple critical point of $g$, there exists a function $\tilde{p}$, analytic in a neighborhood of $t_0$, such that

\begin{equation} \label{p and p tilde}
p(t) := g(t_0) - g(t) = (t-t_0)^2 \tilde{p}(t), \qquad \qquad \tilde{p}(t_0) = -\frac{g''(t_0)}2 \neq 0. 
\end{equation}

\n We are now ready to state the main result, whose proof is given in the Appendix.

\begin{theorem} \label{THM 1}
Assume (H1)-(H5), then using the notation introduced above, as $\lambda \to + \infty$,

\begin{equation*}
I(\lambda) \sim \frac{i e^{\lambda g(t_0)}}{\lambda^{3/2}} \sum_{N=0}^\infty \frac{\CR_N}{\lambda^N}. 
\end{equation*}

\n The coefficients in this expansion are given by

\begin{equation} \label{R_n coeff}
\CR_N = \sum_{n=0}^N \frac{\Gamma \! \left( n + \frac{1}{2} \right)}{(N-n)!} R_n\left[ \partial_\epsilon^{N-n} G(t, 0) \right],
\end{equation}

\n where $\Gamma$ denotes the Gamma function and $R_n$ is the residue operator at $t_0$ defined by

\begin{equation} \label{R_n}
R_n[f] = \operatorname*{Res}_{t=t_0}\frac{f(t)}{p(t)^{n+\frac12}}.
\end{equation}

\n In particular, if the incident wave $v$ is nonscattering, then $\CR_N = 0$ for all $N=0,1,...$.
\end{theorem}

The formula \eqref{R_n} uses square root of the function $p$. The appropriate choice of branch is discussed in Remark~\ref{REM branch} below. Thus, we have obtained an explicit formula for the coefficients $\CR_N$ in terms of the $\epsilon$-derivatives and $t_0$-residues of the function $G$. It is not hard to show that the the residue in the above formula corresponds to a pole of order at most $2n+1$, and hence can be computed using derivatives of order up to $2n$. Consequently, $\CR_N$ can alternatively be written in closed form using only $\epsilon$- and $t$-derivatives of $G$ evaluated at $(t,\epsilon)=(t_0,0)$ (see Remark~\ref{REM residue} below).  

The reduced form of the nonscattering conditions, namely Theorem~\ref{THM reduced}, is obtained from the above theorem, and more specifically from the closed-form expressions for the coefficients $\CR_N$. These coefficients can be computed for arbitrary $N$, after which the reduction procedure described in Section~\ref{SECT proof of reduction} yields the corresponding operator $L_N$ explicitly. 

\begin{remark}[Critical points of $h$] \mbox{} \label{REM h crit point 2}
\normalfont

\n We can also consider critical points of $h$ by choosing the other branch of the square root in the definition of the test vector $\bm{\zeta}(\lambda)$ in \eqref{I cali}. Namely, $\bm{\zeta}(\lambda) = (-i\lambda, - \sqrt{\lambda^2 + \kappa^2})$. With this choice, proceeding as in the derivation of \eqref{I cali G g} from \eqref{I cali}, we obtain

\begin{equation*}
I(\lambda) =  -\frac{i}{\lambda} \int_{-\pi}^\pi H(t, \lambda^{-1}) e^{\lambda h(t) }dt,
\end{equation*}

\n where $H$ is obtained from $G$ by interchanging $A$ with $B$, and $g$ with $h$ (in particular $x_2 = (g-h)/2i$ gets negated). This explains the formulas stated in Remark~\ref{REM h crit point}.
\end{remark}

\begin{remark}[Branch of the square root] \label{REM branch} \mbox{}
\normalfont 

\n The square root of the function $p$ is defined by 

\begin{equation*}
p(t)^\frac{1}{2} = (t-t_0) \tilde{p}(t)^\frac{1}{2}.
\end{equation*}

\n Let us explain the choice of branch for the square root of $\tilde{p}$, see \cite{DLMF}. First, let $\omega$ denote the angle of the tangent to the contour $\PC$ at $t_0$ from the ``right", i.e., $\omega$ is the limiting value of $\arg(t-t_0)$ as $t\to t_0$ along the portion of $\PC$ joining $t_0$ to $\pi$. We may use any continuously varying choice of $\arg$ along the contour, so $\omega$ is defined modulo $2\pi$ (the particular choice will not affect the branch of the square root defined below). Next, choose $\omega_0 = \arg (-g''(t_0))$ so that $|\omega_0 + 2\omega| \leq \pi/2$. We then define $\tilde{p}(t_0)^{\frac{1}{2}}$ using the argument $\omega_0$. Since $\tilde{p}$ is nonzero in a sufficiently small disk centered at $t_0$, it admits an analytic square root there and we choose the branch that takes the prescribed value $\tilde{p}(t_0)^{\frac{1}{2}}$ at $t=t_0$.
\end{remark}

\begin{remark}[Calculating the residue] \label{REM residue} \mbox{}
\normalfont 

\n Suppose that $f(t)$ is analytic in a neighborhood of $t_0$. Then the residue in \eqref{R_n} can be computed using derivatives of order up to $2n$. Indeed,

\begin{equation*}
\Res_{t=t_0}\frac{f}{p^{n+\frac12}}= \Res_{t=t_0}\frac{f\tilde{p}^{-n-\frac12}}{(t-t_0)^{2n+1}} =
\frac{1}{(2n)!}\left.\frac{d^{2n}}{dt^{2n}}\left( \frac{f}{\tilde{p}^{\, n+\frac12}}\right)\right|_{t=t_0} =  \left. P_n \left( \tfrac{d}{dt} \right) f \right|_{t=t_0},
\end{equation*}

\n where in the last step we expanded the derivative using Leibniz's rule and set

\begin{equation*}
P_n(z) = \sum_{j=0}^{2n} c_{n,j} z^{2n-j}, 
\qquad \qquad 
c_{n,j} = \frac{1}{j! (2n-j)!}
\left. \frac{d^{j}}{dt^{j}} \left( \frac{1}{
\tilde{p}^{\, n+\frac12}}\right)
\right|_{t=t_0}.
\end{equation*}

\n As a result, the coefficient $\CR_N$ can be expressed in terms of a linear differential operator $\mathcal{L}_N$ in the variables $t$ and $\epsilon$:

\begin{equation*}
\CR_N = \mathcal{L}_NG(t_0, 0), \qquad \qquad
\mathcal{L}_N = \sum_{n=0}^N \frac{\Gamma \! \left( n + \frac{1}{2} \right)}{(N-n)!} P_n(\partial_t) \partial_\epsilon^{N-n}.
\end{equation*}

\end{remark}

\subsection{Application: ellipse} \label{SECT app ellipse}

\n Let $D$ be the elliptical inhomogeneity

\begin{equation} \label{E_a,b}
E_{a,b} = \left\{(x_1,x_2)\in\mathbb{R}^2:
\frac{x_1^2}{a^2}+\frac{x_2^2}{b^2}<1\right\}.
\qquad a \geq b>0.
\end{equation}

\n We use the parametrization $\bm{x}(t) = \left(a \cos t, b \sin t \right)$ and introduce the focal distance

\begin{equation} \label{c ellipse}
c = \sqrt{a^2-b^2}.
\end{equation}

\n It was shown in \cite{HV1} that this inhomogeneity satisfies all the hypotheses of Section~\ref{SECT main results}. That is, $g=x_1+ix_2$ has a simple critical point and there exists a contour $\PC$ satisfying hypothesis (H4). Since the calculation of the critical point and the associated quantities is concise and will also be needed in establishing our new scattering results, we briefly recall it here, thereby making the discussion self-contained. In addition, we show that $h=x_1-ix_2$ also has a simple critical point and satisfies the corresponding assumptions. It is not hard to show that one can rewrite

\begin{equation*}
g(t)=c\cos(t-t_0), \qquad \qquad h(t)=c\cos(t+t_0),
\end{equation*}

\n where

\begin{equation} \label{t_0 mu_0}
t_0 = i \mu_0, \qquad \qquad \mu_0 = \tanh^{-1} \frac{b}{a}.
\end{equation}

\n It follows immediately that $t_0$ is a simple critical point of $g$, while $\tilde{t}_0 = -t_0$ is a simple critical point of $h$. When $a=b$ the ellipse becomes a circle and $t_0 = i\infty$. Thus, the finite critical point $t_0$ arises only in the case of a true ellipse, $a>b$. To verify the topological assumption concerning the existence of a contour $\PC$ as in (H4), in the left portion of Figure~\ref{FIG el} we plot the curves $\Re g(t)=\Re g(t_0)$ and mark with a ``$-$" the region where $\Re g(t)<\Re g(t_0)$. The goal is to connect $-\pi$ to $\pi$ by a contour $\PC$ that remains entirely within the ``$-$" region, except at the crossing point $t_0$. As the figure shows, such a contour can clearly be constructed. The right portion of Figure~\ref{FIG el} shows the analogous contour construction for the function $h$ and its critical point $-t_0$. Regarding assumption (H5), the minimal region of analyticity required for our theorem to apply is the closure of the domain lying above the real axis and below the curve $\Re g(t)=\Re g(t_0)$. Explicitly,

\begin{equation*} 
\CT_0 = \left\{ t \ : \ |\Re t| \leq \arcsin \frac{b}{a}, \qquad 0\leq \Im t \leq \mu_0 - \tanh^{-1}\left(|\sin \Re t|\right) \right\}.
\end{equation*}

\n The analogous region associated with the function $h$ is the complex conjugate region $\overline{\CT}_0$. We are now ready to apply Theorem~\ref{THM reduced} to the elliptical inhomogeneity. We consider only the first- and second-order conditions associated with the two critical points introduced above, namely \eqref{1st and end cond v} and \eqref{1st and 2nd cond with h}. To rewrite these conditions explicitly, it remains to compute the corresponding points

\begin{equation} \label{x_0 ellipse}
\bm{x}_0 = \bm{x}(t_0) = \left(\frac{a^2}{c}, i\frac{b^2}{c} \right), \qquad \tilde{\bm{x}}_0 = \bm{x}(-t_0) = \left(\frac{a^2}{c}, -i\frac{b^2}{c} \right)
\end{equation}

\n as well as the coefficients

\begin{equation*}
g_1 = g_3 = 0, \qquad g_2 = - c \qquad \text{and} \qquad h_1 = -2i \frac{ab}{c}, \qquad h_2 = - \frac{a^2+b^2}{c}.
\end{equation*}

\n The same formulas hold for $\tilde{h}_j$ in place of $g_j$ and $\tilde{g}_j$ in place of $h_j$. We recall that $g_j=g^{(j)}(t_0)$, while tildes denote evaluation at $\tilde{t}_0=-t_0$, and analogously for $h$.

\begin{figure}[H]
    \centering
    \includegraphics[width=0.49\linewidth]{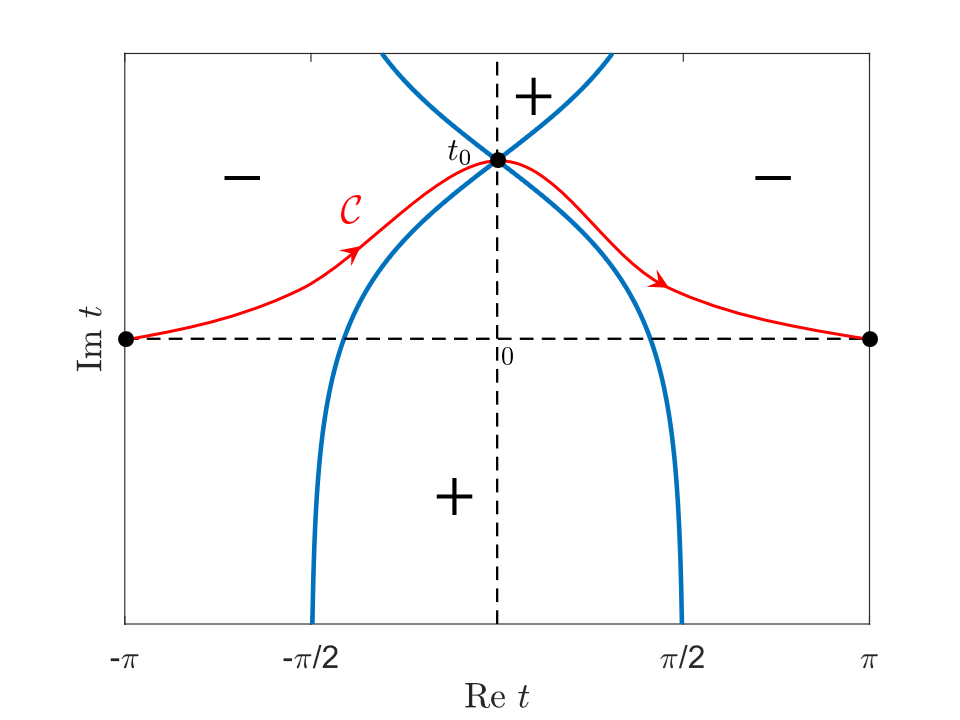}
    \includegraphics[width=0.495\linewidth]{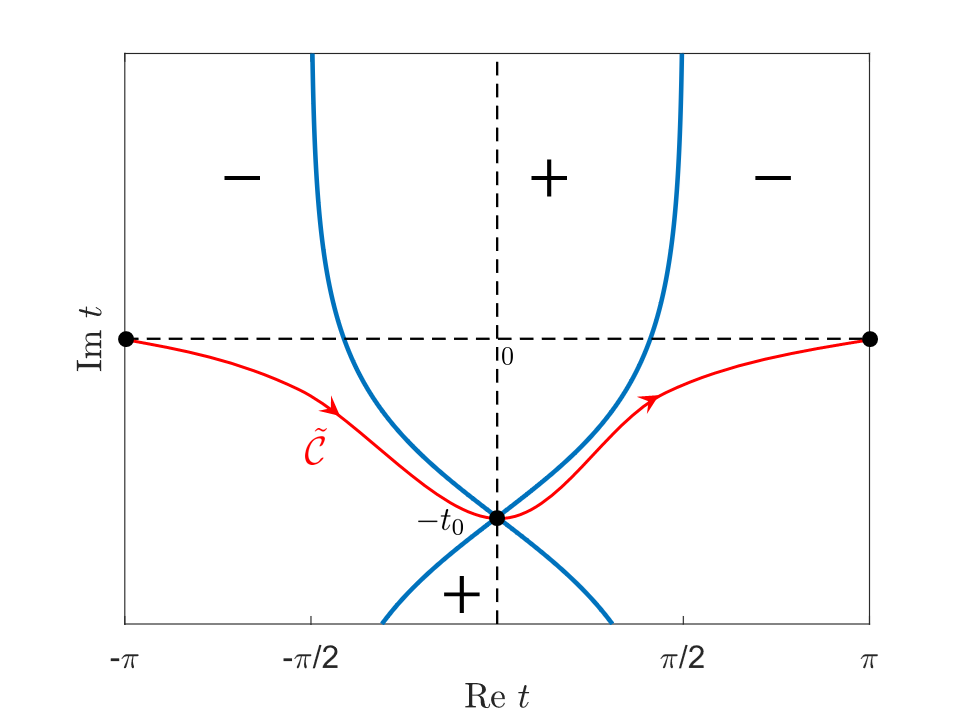}
    \caption{\textbf{Left:} The curve $\Re g(t)=\Re g(t_0)$, with ``$-$" (resp. ``$+$") indicating the region where $\Re g(t)<\Re g(t_0)$ (resp. $\Re g(t)>\Re g(t_0)$). The contour $\PC$ lies entirely in the ``$-$" region, except at the point $t_0$. \textbf{Right:} The analogous construction for $h$ and the critical point $-t_0$.}
    \label{FIG el}
\end{figure}

\n Thus, we obtain:

\begin{corollary} \label{CORO ellipse}
Let $E_{a,b}$ with $a>b$ be the elliptical inhomogeneity \eqref{E_a,b} with constant refractive index $0<q_0 \neq 1$. Suppose that the incident wave $v$ is nonscattering.

\begin{enumerate}
\item[(i)] If $v$ extends to an analytic function of two complex variables in an open set of $\CC^2$ containing $\left\{\bm{x}(t): t \in \CT_0 \right\}$, then

\begin{equation} \label{1st and 2nd ellipse}
v(\bm{x}_0) = 0 \quad \text{and} \quad
3c^3 A v(\bm{x}_0) + 3c(a^2+b^2) B v(\bm{x}_0) + 2a^2b^2 B^2 v(\bm{x}_0) = 0.
\end{equation}

\item[(ii)] If $v$ extends to an analytic function of two complex variables in an open set of $\CC^2$ containing $\left\{\bm{x}(t): t \in \overline{\CT}_0 \right\}$, then

\begin{equation} \label{1st and 2nd ellipse h}
v(\tilde{\bm{x}}_0) = 0
\quad \text{and} \quad
3c^3 B v(\tilde{\bm{x}}_0) + 3c(a^2+b^2) A v(\tilde{\bm{x}}_0) + 2a^2b^2 A^2 v(\tilde{\bm{x}}_0) = 0
\end{equation}

\end{enumerate}
\end{corollary}

\n In particular, if $v$ admits an analytic extension to an open set containing $\left\{\bm{x}(t): t \in \CT_0 \cup \overline{\CT}_0 \right\}$ as is the case, e.g. for Herglotz waves, then both parts of the above corollary apply. Considering only the first-order conditions, we immediately conclude that a nonscattering incident wave $v$ must satisfy $v(\bm{x}_0) = v(\tilde{\bm{x}}_0) = 0$. We may view this as a symmetry-type condition: any $v$ that fails this condition, i.e. that does not vanish simultaneously at the two conjugate points $\bm{x}_0$ and $\tilde{\bm{x}}_0$, necessarily scatters from the ellipse.

\begin{remark}
\normalfont The third-order condition $L_2 v(\bm{x}_0) = 0$ from Remark~\ref{REM 3rd order} now takes the following form (see the Maple script in \cite{maple}):

\begin{equation*}
\begin{split}
L_2 =&  60c^2 A^2 +16k^2\frac{a^2b^2}{c} (4q_0-5) B
+\frac{20}{c^2} \left(3a^4 + 10a^2b^2 + 3b^4\right)B^2+
\\[.05in]
&+80 \frac{a^2b^2}{c^3}(a^2+b^2) B^3
+\frac{16a^4b^4}{c^4}B^4.
\end{split}
\end{equation*}

\end{remark}

\vspace{.05in}

Our next key application concerns incident waves consisting of a single separated mode, which we will refer to as single separated incident modes. These modes are symmetric and the conditions in part $(ii)$ provide no additional information beyond those in part $(i)$. In particular, if such an incident wave vanishes at $\bm{x}_0$, then it also vanishes at $\tilde{\bm{x}}_0$. We will show that every single separated incident mode necessarily scatters, using only the two conditions in part $(i)$. This is the subject of the next subsection.   

\subsubsection{Separated incident modes} \label{SECT separated modes}

\n Let us begin with the case of the disk $D_a=E_{a,a}$ of radius $a$ and recall the classical nonscattering result in this setting. Since the incident wave $v$ solves the Helmholtz equation in the whole space, it also solves it inside the disk. Separating variables in polar coordinates gives 

\begin{equation*}
v(\bm{x}) = \sum_{n\in \ZZ} v_n J_n(kr) e^{in\theta},
\end{equation*}

\n where $v_n$ are complex constants, $(r,\theta)$ are the polar coordinates of $\bm{x}$ and $J_n$ denotes the Bessel function of the first kind of order $n$. It is well known that, for each fixed $n$, the single separated mode $J_n(kr)e^{in\theta}$ is nonscattering at infinitely many wavenumbers \cite{CCH23, VX}. These nonscattering wavenumbers are obtained by separating variables in the transmission problem and are precisely the positive roots of the determinant condition

\begin{equation} \label{det}
k J_n'(k a) J_n(\kappa a) - \kappa  J_n'(\kappa a) J_n(ka) = 0,
\end{equation}

\n where, as before, $\kappa = k \sqrt{q_0}$. More precisely we have the following.

\begin{theorem} \label{THM disk}
Let $D_a=E_{a,a}$ be the disk of radius $a>0$ centered at the origin, with constant refractive index $0<q_0 \neq 1$. For each $n\in\mathbb{Z}$, consider the single separated incident mode $v(k, \bm{x}) = J_n(kr) e^{in\theta}$. Then there exists an infinite sequence of distinct positive wavenumbers $k_1 < k_2<...$, which solve \eqref{det}, such that $v(k_j, \bm{x})$ is nonscattering for every $j=1,2,...$.
\end{theorem}

Continuing our study of the incident waves $v(k,\bm{x})=J_n(kr)e^{in\theta}$, one may ask whether they can be nonscattering for an elliptical inclusion $E_{a,b}$ with $a>b$. This is particularly interesting from a perturbative point of view when $a$ is close to $b$, so that $E_{a,b}$ is close to the disk $D_a$. The result of \cite{VX}, which applies to arbitrary incident Herglotz waves, implies in particular that for the single-mode incident waves considered above, the set of nonscattering wavenumbers for $E_{a,b}$ is at most finite, in contrast to the infinite set arising for $D_a$. We show below that this set is in fact empty.

\begin{theorem} \label{THM ellipse J_n}
Let $E_{a,b}$ with $a>b$ be the elliptical inhomogeneity \eqref{E_a,b} with constant refractive index $0<q_0 \neq 1$. Then for every $n \in \ZZ$ and every $k>0$ the incident wave $v(k, \bm{x}) = J_n(kr) e^{in\theta}$ scatters. 
\end{theorem}

\begin{proof}
In \cite{HV1}, we have already established that these incident waves scatter in the case $n\leq 5$. It therefore remains to establish the result for indices $n>5$, which can be done using the symmetry of the ellipse. Indeed, suppose that $v(k, \bm{x}) = J_n(kr) e^{in\theta}$ is nonscattering for $E_{a,b}$ for some $k>0$ and some $n > 5$. Since $E_{a,b}$ is symmetric with respect to the horizontal axis, it follows that the wave $J_n(kr) e^{-in\theta}$ is also nonscattering. In view of the identity $J_{-n} = (-1)^n J_n$ this implies that the wave $J_{-n}(kr) e^{-in\theta}$ is nonscattering as well. This is a contradiction, since $-n<-5\leq 5$ and the result of \cite{HV1} applies to the index $-n$. The same conclusion can also be obtained using the second-order conditions in \eqref{1st and 2nd ellipse} and \eqref{1st and 2nd ellipse h}.
\end{proof}

The incident waves $v(k,\bm{x})=J_n(kr)e^{in\theta}$ are naturally adapted to the disk geometry, since they arise from separation of variables for the Helmholtz equation in the disk. Therefore, one would not expect these incident waves to be nonscattering for a truly elliptical inhomogeneity, and Theorem~\ref{THM ellipse J_n} confirms this expectation. We therefore turn to the elliptical analogues of the above incident Bessel modes and show that the conclusion of Theorem~\ref{THM disk} fails entirely for a truly elliptical inhomogeneity. First, using elliptic coordinates and separating the variables in the Helmholtz equation inside $E_{a,b}$ we obtain (see Section~\ref{SECT ellipse sep var} for details)

\begin{equation} \label{v Matheiu series}
v(\bm{x}) = \sum_{n=0}^{\infty}
v_n \Ce_n(\mu,q) \ce_n(\nu,q)
+
\sum_{n=1}^{\infty} \tilde{v}_n \Se_n(\mu,q) \se_n(\nu,q),
\end{equation}

\n where $v_n$ and $\tilde{v}_n$ are complex constants and $(\mu,\nu)$ are the elliptic coordinates of $\bm{x}$. Here, $\ce_n$ and $\se_n$ denote the standard (or angular) Mathieu functions, while $\Ce_n$ and $\Se_n$ denote the modified (or radial) Mathieu functions of order $n$. Finally, the parameter

\begin{equation*}
q = \frac{k^2 (a^2-b^2)}{4}
\end{equation*}

\n should not be confused with the refractive index $q_0$. We use the notation $q$ here to remain consistent with the standard convention for Mathieu functions. We now state the main result of this subsection, whose proof is given in Section~\ref{SECT proof ellipse}.

\begin{theorem} \label{THM Mathieu}
Let $E_{a,b}$ with $a>b$ be the elliptical inhomogeneity \eqref{E_a,b} with constant refractive index $0<q_0 \neq 1$. Then, for every $k>0$, each single separated incident mode
\begin{equation*}
v(k,\bm{x})=\Ce_n(\mu,q)\ce_n(\nu,q),
\qquad n=0,1,\ldots,
\end{equation*}
as well as each single separated incident mode
\begin{equation*}
v(k,\bm{x})=\Se_n(\mu,q)\se_n(\nu,q),
\qquad n=1,2,\ldots,
\end{equation*}
scatters from $E_{a,b}$.
\end{theorem}

\section{Proof of Theorem~\ref{THM Mathieu}} \label{SECT proof ellipse}
\setcounter{equation}{0}

In this section, we prove Theorem~\ref{THM Mathieu}. In the first subsection, we introduce elliptic coordinates and separate variables in the Helmholtz equation inside the ellipse, obtaining the ODEs that lead to the Mathieu functions. The second subsection contains the proof of the theorem, where we show that a single separated incident mode cannot satisfy the first- and second-order necessary conditions \eqref{1st and 2nd ellipse} for nonscattering and therefore must scatter.

\subsection{Separation of variables and Mathieu functions} \label{SECT ellipse sep var}

Recall that $c$ denotes the focal distance defined in \eqref{c ellipse}. The elliptic coordinates $(\mu,\nu)$ are defined by

\begin{equation*}
\begin{cases}
x_1=c\cosh\mu\cos\nu
\\
x_2=c\sinh\mu\sin\nu,
\end{cases}
\qquad \qquad \mu \geq 0, \ \nu \in [0,2\pi].
\end{equation*}

\n The curves of constant $\mu$ are confocal ellipses with foci at $(\pm c,0)$, while the curves of constant $\nu$ are hyperbolas. In elliptic coordinates, the ellipse $\partial E_{a,b}$ is given by

\begin{equation*}
\mu=\mu_0, \qquad \qquad \mu_0 = \tanh^{-1} \frac{b}{a}.
\end{equation*}

\n The Laplacian in elliptic coordinates is given by \cite{moon}

\begin{equation*}
\Delta=\frac{2}{c^2\left(\cosh 2\mu-\cos 2\nu \right)}
\left( \partial_\mu^2 + \partial_\nu^2 \right).
\end{equation*}

\n Therefore, the Helmholtz equation \eqref{v Helm} becomes

\begin{equation*}
\partial_\mu^2 v + \partial_\nu^2 v
+ \frac{k^2c^2}{2}
\left( \cosh 2\mu-\cos 2\nu \right)v =0.
\end{equation*}

\n Separating variables $v(\mu,\nu)=X(\mu)Y(\nu)$ leads to the following ODEs:

\begin{equation} \label{X Y}
\begin{cases}
X''(\mu)-\left(\sigma-2q\cosh 2\mu\right)X(\mu)=0,
\\
Y''(\nu)+\left(\sigma-2q\cos 2\nu\right)Y(\nu)=0,
\end{cases}
\end{equation}

\n where $q = k^2c^2/4$ and $\sigma$ denotes the separation constant. The equation for $Y$ is known as the standard (or angular) Mathieu equation. Since $\nu$ is the angular variable, we impose periodic boundary conditions on $Y$, which restrict $\sigma$ to the eigenvalues of the Mathieu equation. There are two families of such eigenvalues, with corresponding eigenfunctions (see Chapter~28 of \cite{DLMF})

\begin{equation*}
\begin{split}
\sigma=a_n(q), \qquad Y(\nu) = \ce_n(\nu, q) \qquad \qquad n=0,1,...,
\\
\sigma=b_n(q), \qquad Y(\nu) = \se_n(\nu, q) \qquad \qquad n=1,2,....
\end{split}
\end{equation*}

\n These two families of functions reduce to cosines and sines when $q=0$, which explains the notation. Note that the equation for $X$ is obtained from that for $Y$ by setting $\nu=i\mu$. The equation for $X$ is known as the modified (or radial) Mathieu equation and the functions corresponding to the two families above are, respectively,

\begin{equation*}
\begin{split}
X(\mu) = \Ce_n(\mu, q) = \ce_n(i\mu, q), \qquad \qquad n=0,1,...,
\\
X(\mu) = \Se_n(\mu, q) = -i \se_n(i\mu, q), \qquad \qquad n=1,2,....
\end{split}
\end{equation*}

\n For fixed $Y$ and $\sigma$, the function $X$ satisfies a second-order ODE which has two linearly independent solutions. However, we are interested in pairs $(X,Y)$ for which the resulting function $v$ is smooth across the focal segment $\mu=0$. This requires $X$ to have the same parity as $Y$, leaving, up to normalization a unique admissible choice for $X$, namely the one given above. This explains why solutions $v$ of the Helmholtz equation inside the ellipse admit the series representation \eqref{v Matheiu series} in terms of Mathieu functions.

\subsection{Proof of Theorem~\ref{THM Mathieu}}

\n In view of the previous subsection, a single separated mode takes the form

\begin{equation*}
v(\bm{x}) = Y(i\mu) Y(\nu),
\end{equation*}

\n where $(\mu, \nu)$ are the elliptic coordinates of $\bm{x}$ and the function $Y(\nu)$ is either $\ce_n(\nu,q)$ for $n \geq 0$, or $\se_n(\nu,q)$ for $n\geq 1$. Recall that the point $\bm{x}_0$ is given by \eqref{x_0 ellipse}. In elliptic coordinates this point has the representation

\begin{equation*}
\bm{x}_0  \longleftrightarrow (\mu_0, i\mu_0).
\end{equation*}

\n Consequently,

\begin{equation*}
v(\bm{x}_0) = Y^2(i\mu_0).
\end{equation*}

\n Suppose that $v$ is nonscattering. Then the first-order nonscattering condition in \eqref{1st and 2nd ellipse} implies that $Y(i\mu_0)=0$. To arrive at the desired contradiction, we will show that the second-order nonscattering condition in \eqref{1st and 2nd ellipse} implies that $Y'(i\mu_0)=0$. Together, these two conditions imply that $Y$ is identically zero: indeed, $Y$ has zero Cauchy data at the point $i\mu_0$ and satisfies a linear ODE with entire coefficients (see \eqref{X Y}).

To rewrite the second-order condition we first express the operators $A$ and $B$ in elliptic coordinates.

\begin{lemma} \label{LEM A B elliptic}

The operators $A = \partial_{x_1}-i\partial_{x_2}$ and $B = \partial_{x_1}+i\partial_{x_2}$ take the following form in elliptic coordinates

\begin{equation*}
A=\frac{1}{c\sinh(\mu+i\nu)}
\left(\partial_\mu-i\partial_\nu\right),
\qquad \qquad
B=
\frac{1}{c\sinh(\mu-i\nu)}
\left(\partial_\mu+i\partial_\nu\right).
\end{equation*}

\end{lemma}

\begin{proof}

\n Let us introduce the auxiliary variables

\begin{equation*}
\xi=x_1+ix_2, \qquad \qquad \eta=x_1-ix_2.
\end{equation*}

\n The chain rule readily implies

\begin{equation} \label{xi eta A B}
2 \partial_\xi =  \partial_{x_1}-i\partial_{x_2}  =
A,
\qquad \qquad
2 \partial_\eta = \partial_{x_1}+i\partial_{x_2} 
= B.
\end{equation}

\n In terms of the elliptic coordinates we have

\begin{equation*}
\xi=c\cosh(\mu+i\nu), \qquad \qquad \eta=c\cosh(\mu-i\nu).
\end{equation*}

\n Our goal is to express the $\mu,\nu$ derivatives in terms of the $\xi,\eta$ derivatives and then use the identities \eqref{xi eta A B} to rewrite the latter in terms of $A$ and $B$. By the chain rule,

\begin{equation*}
\begin{split}
\partial_\mu &= c\sinh(\mu+i\nu)\partial_\xi +
c\sinh(\mu-i\nu)\partial_\eta
\\
\partial_\nu &= ic\sinh(\mu+i\nu)\partial_\xi
- ic\sinh(\mu-i\nu)\partial_\eta.
\end{split}
\end{equation*}

\n Consequently,

\begin{equation*}
\begin{split}
\partial_\mu-i\partial_\nu =
2c\sinh(\mu+i\nu)\partial_\xi = c\sinh(\mu+i\nu)A
\\
\partial_\mu+i\partial_\nu = 2c\sinh(\mu-i\nu)\partial_\eta
= c\sinh(\mu-i\nu)B.
\end{split}
\end{equation*}

\n Solving for $A$ and $B$ yields the desired formulas.
\end{proof}

\n Next, we compute the expressions for $Av$, $Bv$ and $B^2v$ evaluated at $\bm{x}_0$.

\begin{lemma} \label{LEM A B B^2}
Assume that $Y(i\mu_0) = 0$, then

\begin{enumerate}
\item[$(i)$] $\displaystyle Av(\bm{x}_0)=\frac{1}{c}Y'(i\mu_0)^2$

\item[$(ii)$] $\displaystyle Bv(\bm{x}_0) = 0$

\item[$(iii)$] $\displaystyle B^2v(\bm{x}_0) = -\frac{c^2}{2a^2b^2} Y'(i\mu_0)^2$
\end{enumerate}

\end{lemma}

\begin{proof}

$(i)$ Using Lemma~\ref{LEM A B elliptic} we compute

\begin{equation*}
Av(\bm{x})=i\frac{Y'(i\mu)Y(\nu)-Y(i\mu)Y'(\nu)}{c\sinh(\mu+i\nu)}.
\end{equation*}

\n Our goal is to evaluate this expression at $\bm{x}=\bm{x}_0$, or equivalently, in elliptic coordinates at $(\mu,\nu)=(\mu_0,i\mu_0)$, where $\mu_0>0$. At this point, both the numerator and the denominator vanish. A straightforward argument shows that the resulting singularity is removable and that the limit as $(\mu,\nu)\to(\mu_0,i\mu_0)$ exists. We omit these details and compute the limit by fixing $\mu=\mu_0$ and letting $\nu\to i\mu_0$. Thus, using $Y(i\mu_0)=0$, it remains to compute

\begin{equation*}
Av(\bm{x}_0)=\frac{i}{c} Y'(i\mu_0) \lim_{\nu\to i\mu_0}
\frac{Y(\nu)}{\sinh(\mu_0+i\nu)}.
\end{equation*}

\n To evaluate the limit, we expand the numerator about $\nu=i\mu_0$ using the analyticity of $Y$:

\begin{equation*}
Y(\nu)=Y'(i\mu_0)(\nu-i\mu_0)+
O\left((\nu-i\mu_0)^2\right).
\end{equation*}

\n As for the denominator, observe that $\mu_0+i\nu=i(\nu-i\mu_0)$ and hence

\begin{equation*}
\sinh(\mu_0+i\nu)=i(\nu-i\mu_0)+O\left((\nu-i\mu_0)^2\right).
\end{equation*}

\n The last two formulas complete the proof of this part.

\vspace{.1in}

$(ii)-(iii)$ Again invoking Lemma~\ref{LEM A B elliptic} we first compute

\begin{equation*}
Bv(\bm{x}) = i \frac{Y'(i\mu)Y(\nu)+Y(i\mu)Y'(\nu)
}{c\sinh(\mu-i\nu)},
\end{equation*}

\n which readily implies that $Bv(\bm{x}_0)=0$. To find $B^2v$, let us set $S(\mu, \nu) = c\sinh(\mu-i\nu)$, then applying the product rule gives

\begin{equation*}
B^2v(\bm{x}) = \frac{1}{S}\left[\left(\partial_{\mu}+i\partial_{\nu}
\right) \left(\frac{1}{S} \right)\right]
\left(\partial_{\mu}+i\partial_{\nu}\right) \left[ Y(i\mu) Y(\nu) \right]
+
\frac{1}{S^2}\left(\partial_{\mu}+i\partial_{\nu}\right)^2 \left[ Y(i\mu) Y(\nu) \right].
\end{equation*}

\n We will show that upon evaluating at $\bm{x}_0$, which is equivalent to evaluating the elliptic coordinates at $(\mu_0,i\mu_0)$, the first term on the right-hand side of the above expression vanishes. Indeed, this follows from

\begin{equation*}
\left(\partial_{\mu}+i\partial_{\nu}\right) \left[ Y(i\mu) Y(\nu) \right]
=iY'(i\mu)Y(\nu)+iY(i\mu)Y'(\nu).
\end{equation*}

\n Next we simplify

\begin{equation*}
\left(\partial_{\mu}+i\partial_{\nu}\right)^2 \left[ Y(i\mu) Y(\nu) \right] =  -Y''(i\mu)Y(\nu)-2Y'(i\mu)Y'(\nu)
-Y(i\mu)Y''(\nu),
\end{equation*}

\n which, upon evaluation at $(\mu_0,i\mu_0)$ and using $Y(i\mu_0)=0$ reduces to $-2 Y'(i\mu_0)^2$. Substituting these expressions into the formula for $B^2v$ we arrive at

\begin{equation*}
B^2v(\bm{x}_0) = -\frac{2}{c^2 \sinh^2(2\mu_0)} Y'(i\mu_0)^2. 
\end{equation*}

\n To conclude the proof it remains to use the identity $\sinh(2\mu_0) = 2ab/c^2$.

\end{proof}

\n In view of Lemma~\ref{LEM A B B^2}, we have $Bv(\bm{x}_0)=0$. Therefore, the second-order condition in \eqref{1st and 2nd ellipse} becomes 

\begin{equation*}
2a^2b^2  B^2v(\bm{x}_0) +  3c^3 Av(\bm{x}_0) = 0.
\end{equation*}

\n Invoking Lemma~\ref{LEM A B B^2} once again the expression simplifies to

\begin{equation*}
2c^2 Y'(i\mu_0)^2 = 0,
\end{equation*}

\n which yields the desired conclusion $Y'(i\mu_0)=0$.

\section{Proof of Theorem~\ref{THM reduced}} \label{SECT proof of reduction}
\setcounter{equation}{0}

\n In this section, we prove Theorem~\ref{THM reduced}. Our starting point is to show that the coefficient $\CR_N$ given by \eqref{R_n coeff} can be written in the form stated in the following theorem.

\begin{theorem} \label{THM R_N as L_N}
There exists a nonzero constant $c_N$ such that $\CR_N = c_N \CL_N v(\bm{x}_0)$, where $\CL_N$ is a differntial operator of the form 

\begin{equation} \label{CL_N form}
\CL_N = \sum_{n=0}^N a_{N,n}A^n + \sum_{n=1}^{2N}b_{N,n}B^n.
\end{equation} 

\n Here $a_{N,n}$ and $b_{N,n}$ are complex constants and, with the notation introduced in \eqref{g_n h_n},

\begin{equation} \label{a_N, b_N,2N}
a_{N,N} = \frac{(2N+1)!}{N!} g_2^N, \qquad \qquad b_{N, 2N} = h_1^{2N}.
\end{equation}

\end{theorem}

Before proving this theorem, let us show that it implies Theorem~\ref{THM reduced}. Indeed, if the incident wave $v$ is nonscattering, then all the coefficients $\CR_N$ must vanish, and hence so must all the quantities $\CL_Nv(\bm{x}_0)$. Fix $N>1$ and write the equation $\CL_Nv(\bm{x}_0)=0$ explicitly as

\begin{equation} \label{CL_N=0}
\sum_{n=0}^N a_{N,n}A^n v(\bm{x}_0)
+\sum_{n=1}^{2N}b_{N,n}B^n v(\bm{x}_0)=0.
\end{equation}

\n The idea is simple: we use the preceding equations $\CL_jv(\bm{x}_0)=0$, $j<N$ recursively to eliminate all terms $A^jv(\bm{x}_0)$ with $j<N$. Indeed, the equation corresponding to $j=0$ gives $v(\bm{x}_0)=0$, so the zeroth-order term in \eqref{CL_N=0} vanishes. Next, since $a_{1,1}\neq0$, the equation $\CL_1v(\bm{x}_0)=0$ allows us to express $Av(\bm{x}_0)$ as a linear combination of the terms $B^\ell v(\bm{x}_0)$, $\ell=1,2$. Substituting this expression into \eqref{CL_N=0} eliminates the term involving $Av(\bm{x}_0)$. Continuing recursively, at the final step we express $A^{N-1}v(\bm{x}_0)$ as a linear combination of the terms $B^\ell v(\bm{x}_0)$, $\ell=1,\ldots,2(N-1)$ and arrive at the equation

\begin{equation} \label{CL_N reduced}
a_{N,N}A^N v(\bm{x}_0)+\sum_{n=1}^{2N}\tilde{b}_{N,n}B^n v(\bm{x}_0)=0.
\end{equation}

\n It remains to observe that these substitutions do not alter the coefficient of the highest-order term $B^{2N}$. Consequently, $\tilde{b}_{N,2N} = b_{N,2N}$. The above equation is therefore precisely $L_Nv(\bm{x}_0)=0$, where $L_N$ is given by \eqref{L_N form}. This concludes the proof of Theorem~\ref{THM reduced}.


\subsection{Proof of Theorem~\ref{THM R_N as L_N}} \label{SECT proof of R_N as L_N}
    
The quantity $\CR_N$ in \eqref{R_n coeff} is expressed in terms of residues involving the functions $\partial_\epsilon^n G(t,0)$ for various indices $n$. It will therefore be useful to rewrite these partial derivatives in a more convenient form. From the definition of $G$ in \eqref{G}, its $\epsilon$-derivatives can be expressed in terms of the functions

\begin{equation} \label{alpha_n beta_n}
\alpha_n(t)=\left.\partial_\epsilon^n \left( e^{i\rho(\epsilon)x_2(t)} \right) \right|_{\epsilon=0},
\qquad \qquad
\beta_n(t)=\left.\partial_\epsilon^n
\left( \bigl(\kappa^2-\rho(\epsilon)^2\bigr)
e^{i\rho(\epsilon)x_2(t)}\right)\right|_{\epsilon=0}.
\end{equation}

\n We also introduce

\begin{equation} \label{eta}
\eta(t) = \frac{h(t) - h(t_0)}{2}.
\end{equation}

\begin{remark}
\normalfont This section is independent of Sections~\ref{SECT app ellipse} and \ref{SECT proof ellipse}. Accordingly, we reuse the notation $\eta$ (and, later in the section, $\xi$ and $\mu$) for different purposes. This should cause no confusion, since the section is self-contained.
\end{remark}

\n The next lemma is a key result showing that $\CR_N$ can be expressed as a linear combination of mixed terms $A^{j+1}B^\ell$ satisfying $2j+\ell\leq 2N+1$. Moreover, the coefficients in this representation are given by explicit formulas involving residues of the auxiliary functions introduced above. We also recall that $p(t)=g(t_0)-g(t)$. We postpone the proof of the lemma until the end of the section and first show how it completes the proof of Theorem~\ref{THM R_N as L_N}. 

\begin{lemma} \label{LEM R_N main}
We have the representation $\displaystyle \CR_N = \CR_N^{I} + \CR_N^{II}$, where

\begin{equation} \label{R_N I and II}
\CR_N^{I} = \sum_{n=0}^N \hspace{-.05in} \sum_{\substack{j, \ell \geq 0 \\[.05in] 2j+\ell \leq 2n+1}} \hspace{-.1in} \Phi_{n}(j, \ell) A^{j+1}B^\ell v(\bm x_0),
\qquad \qquad
\CR_N^{II} = \sum_{n=0}^N \hspace{-.05in} \sum_{\substack{j, \ell \geq 0 \\[.05in] 2j+\ell \leq 2n}} \hspace{-.1in} \Psi_{n}(j, \ell) A^{j}B^\ell v(\bm x_0) .
\end{equation}

\n The coefficients are given by the following formulas

\begin{equation*}
\Phi_{n}(j, \ell) = \left( n + \tfrac{1}{2} \right) c_n(j,\ell) \,
\Res_{t=t_0} \frac{\alpha_{N-n} \,  \eta^\ell \, p'}{p^{n -j + \frac{3}{2}}},  
\qquad \qquad
\Psi_{n}(j, \ell) = \tfrac{1}{2} c_n(j,\ell) \,
\Res_{t=t_0}\frac{\beta_{N-n} \, \eta^\ell \, h'}{p^{n-j+\frac12}},
\end{equation*}

\n where

\begin{equation*}
c_n(j,\ell) = \frac{(-1)^j \, \Gamma \! \left( n + \tfrac{1}{2} \right)}{2^{j} \, (N-n)! \, j! \, \ell!}.
\end{equation*}
    
\end{lemma}

\vspace{.1in}

\n \textbf{Reduction of mixed terms.} Recall that $v$ satisfies the Helmholtz equation, and hence $ABv=BAv=-k^2v$ (see \eqref{Helmholtz AB}). Therefore, every mixed term in \eqref{R_N I and II} can be reduced as follows:

\begin{equation*}
A^jB^\ell v
=
\begin{cases}
(-k^2)^\ell A^{j-\ell}v \quad
& j\ge\ell,\\[1ex]
(-k^2)^jB^{\ell-j}v
& \ell\ge j.
\end{cases}
\end{equation*}

\n Clearly, $\CR_N^I$ contains higher powers of $A$ and $B$. We therefore consider a generic term $A^{j+1}B^\ell v$ and distinguish two cases:

\begin{enumerate}
\item[$\bullet$] If $\ell \geq j+1$, then this term reduces to $B^{\ell-j-1} v$, and we have $\ell - j -1 \leq 2N$.

\item[$\bullet$] If $\ell \leq j+1$, then this term reduces to $A^{j+1-\ell} v$, and we have $j+1-\ell \leq N+1$
\end{enumerate}

\n Thus, the powers of $B$ can reach order $2N$, while those of $A$ can reach order $N+1$. Consequently, Lemma~\ref{LEM R_N main} yields

\begin{equation*}
\CR_N = \CL_N v (\bm{x}_0),
\end{equation*}

\n where $\CL_N$ is a linear differential operator of the form

\begin{equation*}
\CL_N = \sum_{n=0}^{N+1} a_{N,n}A^n + \sum_{n=0}^{2N}b_{N,n}B^n.
\end{equation*}

\n Here, $a_{N,n}$ and $b_{N,n}$ are complex constants. We slightly abuse notation by using the same symbols for these coefficients as for those of the operator $\CL_N$ in \eqref{CL_N form} of Theorem~\ref{THM R_N as L_N}. The two sets of coefficients differ only by a common multiplicative constant $c_N$, see \eqref{c_N} below. We now note that, since there are two constants corresponding to the $n=0$ terms, we may normalize $b_{N,0}=0$. To complete the proof, it remains to show that $a_{N,N+1}=0$, so that the highest power of $A$ is at most $N$, and to compute $a_{N,N}$ and $b_{N,2N}$. We do this next.

\vspace{.05in}

\n \textbf{The coefficient $a_{N,N+1}$.} In view of Lemma~\ref{LEM R_N main}, the highest power $A^{N+1}$ arises only from $\CR_N^I$, specifically from the summand with $n=N$ and the term $A^{j+1}B^\ell$ corresponding to $j=N$ and $\ell=0$. Consequently, its coefficient is given by

\begin{equation*}
a_{N,N+1} = \Phi_N(N,0) =  \left( N + \tfrac{1}{2} \right) c_N(N, 0) \Res_{t=t_0} \frac{p'}{p^{\frac{3}{2}}},
\end{equation*}

\n where in the last step we used that $\alpha_0(t)=1$. But note that
\[
\frac{p'}{p^{3/2}}=-2\left(p^{-1/2}\right)'
\]
is a total derivative and therefore its residue vanishes. Hence, $a_{N,N+1}=0$.

\vspace{.1in}

\n \textbf{The coefficient $b_{N,2N}$.} In view of Lemma~\ref{LEM R_N main}, the highest power $B^{2N}$ receives contributions from both $\CR_N^I$ and $\CR_N^{II}$, in each case from the summand with $n=N$. The contribution from $\CR_N^I$ arises from the term $A^{j+1}B^\ell$ with $j=0$ and $\ell=2N+1$, namely $AB^{2N+1}$, which reduces by the Helmholtz equation to $-k^2B^{2N}$. The contribution from $\CR_N^{II}$ arises from the term $A^jB^\ell$ with $j=0$ and $\ell=2N$. Adding these two contributions, we obtain the coefficient

\begin{equation*}
b_{N,2N} = -k^2 \Phi_N(0, 2N+1) + \Psi_N(0, 2N).
\end{equation*}

\n Using the definitions of the coefficients $\Phi$ and $\Psi$ from Lemma~\ref{LEM R_N main}, together with $\alpha_0=1$ and $\beta_0=\kappa^2$, the above equation simplifies to

\begin{equation*}
b_{N,2N} = \frac{\Gamma \left( N + \tfrac{1}{2} \right)}{2 \, (2N)!} \left( -k^2 \Res_{t=t_0} \frac{\eta^{2N+1} p'}{p^{N+\frac32}} + \kappa^2 \Res_{t=t_0} \frac{\eta^{2N} h'}{p^{N+\frac12}}  \right).
\end{equation*}

\n These residues can be computed using Taylor expansions. To this end, set $\tau=t-t_0$. Then, with $h_1=h'(t_0)$, we have

\begin{equation} \label{eta Taylor}
\eta(t) = \frac{h(t)-h(t_0)}{2} = \frac{h_1}{2} \tau + O(\tau^2), \qquad \qquad
h'(t) = h_1  + O(\tau).
\end{equation}

\n Further, recall that $\tilde{p}$ is defined by the identity $p(t)=\tau^2\tilde{p}(t)$ and therefore

\begin{equation} \label{p p ' Taylor}
p(t) = \tilde{p}_0 \tau^2 + O(\tau^3)
\qquad \qquad
p'(t) = 2 \tilde{p}_0 \tau + O(\tau^2),
\end{equation}

\n where we have set $\tilde{p}_0 = \tilde{p}(t_0)$. These expansions readily imply

\begin{equation*}
\Res_{t=t_0} \frac{\eta^{2N+1} p'}{p^{N+\frac32}} = \frac{2 \left(h_1/2\right)^{2N+1}}{\tilde{p}_0^{N+\frac{1}{2}}}, 
\qquad \qquad
\Res_{t=t_0} \frac{\eta^{2N} h'}{p^{N+\frac12}} = \frac{h_1 \left(h_1/2\right)^{2N}}{\tilde{p}_0^{N+\frac{1}{2}}}.
\end{equation*}

\n The two residues are therefore equal, and we arrive at

\begin{equation} \label{b_N,2N}
b_{N,2N} = (\kappa^2 - k^2) \frac{\Gamma \left( N + \tfrac{1}{2} \right)}{(2N)!} \frac{ \left(h_1/2\right)^{2N+1}}{\tilde{p}_0^{N+\frac{1}{2}}}.
\end{equation}

\vspace{.1in}

\n \textbf{The coefficient $a_{N,N}$.} Again using Lemma~\ref{LEM R_N main}, we see that $A^N$ arises in $\CR_N^{II}$ only from the summand with $n=N$, specifically from the term $A^jB^\ell$ with $j=N$ and $\ell=0$. This contribution has coefficient $\Psi_N(N,0)$. In $\CR_N^I$, the term $A^N$ arises from three different summands.

\begin{enumerate}
\item[$\bullet$] The summand with $n=N$, from the term $A^{j+1}B^\ell$ with $j=N-1$ and $\ell=0$. This contribution has coefficient $\Phi_N(N-1,0)$. However, this coefficient vanishes by the same argument used to show that $a_{N,N+1}=\Phi_N(N,0)=0$. Indeed, $\Phi_N(N-1,0)$ is a multiple of the residue of $p'/p^{5/2}$, which is a total derivative and hence has zero residue.

\item[$\bullet$] The summand with $n=N$, from the term $A^{j+1}B^\ell$ with $j=N$ and $\ell=1$. In this case the term is $A^{N+1}B$ which reduces, by the Helmholtz equation, to $-k^2A^N$. Thus, this contribution has coefficient $-k^2\Phi_N(N,1)$.

\item[$\bullet$] The summand with $n=N-1$, from the term $A^{j+1}B^\ell$ with $j=N-1$ and $\ell=0$. This contribution has coefficient $\Phi_{N-1}(N-1,0)$.
\end{enumerate}

\n Collecting all of these contributions, we obtain

\begin{equation*}
a_{N,N} = \Psi_N(N, 0) -k^2 \Phi_N(N, 1) + \Phi_{N-1}(N-1, 0).
\end{equation*}

\n Using the definitions of the coefficients $\Phi$ and $\Psi$ from Lemma~\ref{LEM R_N main}, together with the identity 
$$\left(N-\tfrac{1}{2}\right) \Gamma\! \left(N-\tfrac{1}{2}\right) = \Gamma \! \left(N+\tfrac{1}{2}\right)$$ 

\n we obtain

\begin{equation} \label{a_N,N inter}
a_{N,N} = \frac{(-1)^N \, \Gamma \! \left( N + \tfrac{1}{2} \right)}{2^N \, N!} \left[ \frac{\kappa^2}{2} \Res_{t=t_0} \frac{h'}{p^{\frac{1}{2}}} - k^2 \left( N + \tfrac{1}{2} \right) \Res_{t=t_0} \frac{\eta \, p'}{p^{\frac{3}{2}}} - 2N \Res_{t=t_0} \frac{\alpha_1 \, p'}{p^{\frac{3}{2}}}  \right]. 
\end{equation}

\n Using the Taylor expansions \eqref{eta Taylor} and \eqref{p p ' Taylor} we find

\begin{equation} \label{aaa}
\Res_{t=t_0} \frac{h'}{p^{\frac{1}{2}}} = \Res_{t=t_0} \frac{\eta \, p'}{p^{\frac{3}{2}}} = \frac{h_1}{ \tilde{p}_0^{\frac{1}{2}}}.
\end{equation}

\n It remains to compute the last residue in \eqref{a_N,N inter}. To this end, using the identity $\rho'(0)=\kappa^2/2$ in the definition of $\alpha_1$, we obtain

\begin{equation*}
\alpha_1(t) = \frac{i \kappa^2}{2} x_2(t).
\end{equation*}

\n On the other hand, recalling that $g=x_1+ix_2$, $h=x_1-ix_2$ and $p=g_0-g$, with $g_0=g(t_0)$ constant, we may express

\begin{equation*}
-\frac{1}{2} p(t) - \eta(t) = i \left(x_2(t) - x_2(t_0) \right).
\end{equation*}

\n As a result

\begin{equation*}
\alpha_1(t) = - \frac{\kappa^2}{4} p(t) - \frac{\kappa^2}{2} \eta(t) + \frac{i\kappa^2}{2} x_2(t_0). 
\end{equation*}

\n The first and last terms do not contribute to the last residue in \eqref{a_N,N inter} involving $\alpha_1$. Indeed, as seen above, they give rise to residues of total derivatives, which vanish. Consequently,

\begin{equation*}
\Res_{t=t_0} \frac{\alpha_1 \, p'}{p^{\frac{3}{2}}} = -\frac{\kappa^2}{2} \Res_{t=t_0} \frac{\eta \, p'}{p^{\frac{3}{2}}}.
\end{equation*}

\n The residue on the right-hand side has already been computed in \eqref{aaa}. Combining the above identities, we arrive at

\begin{equation*}
a_{N,N} = \frac{(-1)^N \, \Gamma \! \left( N + \tfrac{1}{2} \right)}{2^N \, N!} \, \frac{h_1}{ \tilde{p}_0^{\frac{1}{2}}} \left[ \frac{\kappa^2}{2} - k^2 \left( N + \tfrac{1}{2} \right) + N \kappa^2 \right].
\end{equation*}

\n The expression in square brackets can be simplified further by factoring out the term $\kappa^2-k^2$. This yields

\begin{equation} \label{a_N,N}
a_{N,N} = \left( \kappa^2-k^2 \right) \left( N + \tfrac{1}{2} \right)  \frac{(-1)^N \, \Gamma \! \left( N + \tfrac{1}{2} \right) h_1}{2^N \, N! \tilde{p}_0^{\frac{1}{2}}} 
\end{equation}

\vspace{.1in}

In summary, we have established that $\CR_N = \CL_N v (\bm{x}_0)$, where the differential operator $\CL_N$ takes the form

\begin{equation*}
\CL_N = \sum_{n=0}^{N} a_{N,n}A^n + \sum_{n=1}^{2N}b_{N,n}B^n.
\end{equation*}

\n and the coefficients $a_{N,N}$ and $b_{N,2N}$ are given by \eqref{a_N,N} and \eqref{b_N,2N}, respectively. To conclude the proof of Theorem~\ref{THM R_N as L_N} and obtain the normalization for which these coefficients take the simpler forms in \eqref{a_N, b_N,2N}, it remains only to observe that the condition $\CL_Nv(\bm{x}_0)=0$ is unchanged if $\CL_N$ is multiplied by any nonzero constant. We may therefore normalize by canceling the factor

\begin{equation} \label{c_N}
c_N = \frac{\Gamma \! \left( N + \tfrac{1}{2} \right)}{2^{2N+1} \, (2N)!} \, \frac{h_1}{\tilde{p}_0^{N+\frac{1}{2}}}
\end{equation}

\n from the formulas \eqref{a_N,N} and \eqref{b_N,2N}. Using the identity $\tilde{p}_0=-g_2/2$ from \eqref{p and p tilde}, the coefficients then reduce to the simpler forms given in \eqref{a_N, b_N,2N}. It remains only to prove Lemma~\ref{LEM R_N main}, which we do in the next subsection.

\subsection{Proof of Lemma~\ref{LEM R_N main}}

Let us set $r=N-n$, and let $\alpha_r$ and $\beta_r$ be defined by \eqref{alpha_n beta_n}. In addition, introduce the functions

\begin{equation} \label{phi psi}
\varphi_r(t) = \alpha_r(t)Av(\bm x(t)), \qquad \qquad \psi_r(t) =
\frac12 h'(t) \beta_r(t) v(\bm x(t)).
\end{equation}

\n Then, from the definition of $G$ in \eqref{G} and the fact that $\partial_\epsilon$ commutes with $\partial_t$, we obtain the representation 

\begin{equation} \label{G partial r}
\partial_\epsilon^rG(t, 0) = \varphi_r'(t) + \psi_r(t),
\end{equation}

\n therefore \eqref{R_n coeff} implies that $\CR_N = \CR_N^I + \CR_N^{II}$, where

\begin{equation} \label{R_N I and II inter}
\CR_N^I = \sum_{n=0}^N \frac{1}{r!} \Gamma \! \left( n + \tfrac{1}{2} \right) R_{n}[ \varphi'_r] ,
\qquad \qquad
\CR_N^{II} = \sum_{n=0}^N \frac{1}{r!} \Gamma \! \left( n + \tfrac{1}{2} \right) R_n[ \psi_r ].
\end{equation}

\n Here, recall that $R_n$ denotes the residue operator defined in \eqref{R_n}. Thus, our goal is to analyze these residues. To this end, since the functions $\phi_r$ and $\psi_r$ involve $v(\bm{x}(t))$ and $Av(\bm{x}(t))$, our first step is to expand these functions in Taylor series.

\vspace{.1in}

\n \textbf{Taylor expansion.} Expanding $v$ and $Av$ in Taylor series about the point $\bm{x}_0$, we obtain, for $\mu=0,1$,

\begin{equation} \label{A mu}
A^\mu v(\bm x(t)) = \sum_{m=0}^{\infty} \frac1{m!}
\left[ (\bm x(t)-\bm x_0)\cdot \nabla \right]^m A^\mu v(\bm x_0).
\end{equation}

\n To rewrite the above derivatives in a convenient form, we recall the definition of the function $\eta$ and introduce, in addition, the function $\xi$ as follows:

\begin{equation*}
\xi(t)= \frac{g(t)-g(t_0)}2,
\qquad \qquad
\eta(t) = \frac{h(t)-h(t_0)}2,
\end{equation*}

\n where, as before, $g=x_1+ix_2$ and $h=x_1-ix_2$. Rewriting the partial derivatives in terms of the operators $A$ and $B$ via

\begin{equation*}
\partial_{x_1} = \frac{B+A}{2}, \qquad \qquad \partial_{x_2} = \frac{B-A}{2i}
\end{equation*}

\n we obtain

\begin{equation*}
(\bm x(t)-\bm x_0)\cdot \nabla = \xi(t) A + \eta(t) B.
\end{equation*}

\n Substituting the latter expression into \eqref{A mu} and applying the binomial formula, we obtain

\begin{equation*}
A^\mu v(\bm x(t)) = \sum_{m=0}^{\infty} \frac1{m!}
\sum_{\substack{j,\ell\ge0\\j+\ell=m}} \binom{m}{j}
\xi(t)^j\eta(t)^\ell A^{\mu + j}B^\ell v(\bm x_0),
\end{equation*}

\n or equivalently,

\begin{equation} \label{Taylor}
A^\mu v(\bm x(t)) = \sum_{j,\ell \geq 0} 
\frac{\xi(t)^j\eta(t)^\ell}{j!\,\ell!} A^{\mu +j}B^\ell v(\bm x_0). 
\end{equation}

\vspace{.1in}

\n \textbf{The residues.} Let us analyze the residues in \eqref{R_N I and II inter}, beginning with

\begin{equation} \label{R_n psi}
R_n[\psi_r] = \frac{1}{2} \Res_{t=t_0}\frac{h'(t) \beta_r(t) v(\bm{x}(t))}{p(t)^{n+\frac12}} = \sum_{j,\ell \geq 0} \tilde{\Psi}_{n}(j, \ell) A^{j}B^\ell v(\bm x_0),  
\end{equation}

\n where we have used the Taylor expansion \eqref{Taylor} with $\mu=0$ and introduced the notation

\begin{equation*}
\tilde{\Psi}_{n}(j, \ell) = \frac{1}{2 j! \ell!} \Res_{t=t_0}\frac{h' \beta_r \xi^j \eta^\ell}{p^{n+\frac12}} = \frac{(-1)^j}{2^{j+1} j! \ell!} \Res_{t=t_0}\frac{h' \beta_r \eta^\ell}{p^{n-j+\frac12}}. 
\end{equation*}

\n In the last step, we used the identity $\xi=-p/2$, which follows directly from the definitions of $\xi$ and $p$. Now, $p$ has a zero of order two at $t_0$, while $\eta$ has a zero of order one, and both $h'$ and $\beta_r$ are analytic in a neighborhood of $t_0$. Therefore, for the above residue to be nonzero, it is necessary that $2n-2j+1-\ell \geq 1$, or equivalently $2j + \ell \leq 2n$. Substituting \eqref{R_n psi} into the formula for $\CR_N^{II}$ in \eqref{R_N I and II inter}, we readily obtain the desired formula from Lemma~\ref{LEM R_N main} upon setting

\begin{equation*}
\Psi_{n}(j, \ell) = \frac{1}{r!} \Gamma \! \left( n + \tfrac{1}{2} \right) \tilde{\Psi}_{n}(j, \ell)
\end{equation*}

\n and recalling that $r=N-n$.

Let us now turn to the term $R_n[\varphi_r']$. To simplify it, we use the fact that the residue of a total derivative vanishes. Therefore,

\begin{equation*}
0 = \Res_{t=t_0} \left( \frac{\varphi_r}{p^{n+\frac12}} \right)' = \Res_{t=t_0} \frac{\varphi_r' p^{n+\frac12} - \left( n+\frac12\right) \varphi_r \, p^{n-\frac12} p' }{p^{2n+1}},
\end{equation*}

\n which implies

\begin{equation} \label{R_n phi prime}
R_n[\varphi_r'] = \Res_{t=t_0} \frac{\varphi_r'}{p^{n+\frac12}} = \left( n+\tfrac{1}{2}\right) \Res_{t=t_0} \frac{\varphi_r p'}{p^{n+\frac32}}. 
\end{equation}

\n Using $\varphi_r(t)=\alpha_r(t)Av(\bm{x}(t))$ and substituting the Taylor expansion of $Av$ from \eqref{Taylor} with $\mu=1$, we obtain

\begin{equation} \label{R_n phi_r'}
R_n[\varphi_r'] = \sum_{j,\ell \geq 0} \tilde{\Phi}_{n}(j, \ell) A^{j+1}B^\ell v(\bm x_0),
\end{equation}

\n where

\begin{equation*}
\tilde{\Phi}_{n}(j, \ell) = \frac{n+\frac{1}{2}}{j! \ell!} \Res_{t=t_0} \frac{\alpha_r \xi^j \eta^\ell p'}{p^{n + \frac{3}{2}}} = \left( n+\tfrac{1}{2}\right) \frac{(-1)^j}{2^j j! \ell!} \Res_{t=t_0} \frac{\alpha_r  \eta^\ell p'}{p^{n -j + \frac{3}{2}}}. 
\end{equation*}

\n Analogously, for the above residue to be nonzero it is necessary that $2n-2j+2- \ell \geq 1$ or equivalently, $2j+\ell \leq 2n+1$. It remains to substitute \eqref{R_n phi_r'} into the formula for $\CR_N^I$ in \eqref{R_N I and II inter}. This yields the desired formula in Lemma~\ref{LEM R_N main} upon setting

\begin{equation*}
\Phi_{n}(j, \ell) = \frac{1}{r!} \Gamma \! \left( n + \tfrac{1}{2} \right) \tilde{\Phi}_{n}(j, \ell).
\end{equation*}

\subsection{The coefficient $\CR_1$ and operator $L_1$} \label{SECT R_1}

Let us simplify the equation $\CR_1=0$ under the assumption that $\CR_0=0$. The latter is equivalent to $v(\bm{x}_0)=0$. By Lemma~\ref{LEM R_N main} and the analysis in Section~\ref{SECT proof of R_N as L_N}, we have already established that $\CR_1 = \CL_1 v(\bm{x}_0)$, where

\begin{equation*}
\CL_1 = a_{1,0} + a_{1,1} A + b_{1,1} B + b_{1,2} B^2.
\end{equation*}

\n Since $v(\bm{x}_0) = 0$, we may disregard the zeroth-order term $a_{1,0}$. Moreover, the coefficients $a_{1,1}$ and $b_{1,2}$ have already been computed in \eqref{a_N,N} and \eqref{b_N,2N}, respectively. More precisely,

\begin{equation*}
a_{1, 1} = - \left( \kappa^2-k^2 \right)  \frac{3 h_1 \Gamma \! \left( \tfrac{3}{2} \right)}{4 \tilde{p}_0^{\frac{1}{2}}},
\qquad \qquad
b_{1, 2} = \left( \kappa^2-k^2 \right)  \frac{h_1^3 \Gamma \! \left( \tfrac{3}{2} \right)}{16 \tilde{p}_0^{\frac{3}{2}}}.
\end{equation*}

\n Therefore, it remains only to compute the coefficient $b_{1,1}$. This can be done in exactly the same way as for the coefficients above, using the representation of $\CR_1$ from Lemma~\ref{LEM R_N main}. Specifically, we have

\begin{equation*}
b_{1,1} = \Psi_1(0,1) - k^2 \Phi_1(0, 2).
\end{equation*}

\n Using the definitions of $\Phi$ and $\Psi$ from Lemma~\ref{LEM R_N main}, the above expression becomes

\begin{equation*}
b_{1,1} = \frac{\Gamma \! \left( \tfrac{3}{2} \right)}{2} \left[ \kappa^2 \Res_{t=t_0} \frac{\eta h'}{p^{\frac{3}{2}}} - \frac{3k^2}{2} \Res_{t=t_0} \frac{\eta^2 p'}{p^\frac{5}{2}} \right].
\end{equation*}

\n First, we establish a simple relation between the two residues above. Using the product rule, we have

\begin{equation*}
\left( \frac{\eta^2}{p^\frac{3}{2}} \right)' = \frac{2\eta \eta'}{p^\frac{3}{2}} - \frac{3 \eta^2 p'}{2 p^\frac{5}{2}}.
\end{equation*}

\n Let us now take residues at $t=t_0$. The residue of the left-hand side vanishes, since it is a total derivative. Moreover, $2\eta'=h'$ by \eqref{eta}. Thus,

\begin{equation*}
\frac{3}{2} \Res_{t=t_0} \frac{\eta^2 p'}{p^\frac{5}{2}} = \Res_{t=t_0} \frac{\eta h'}{p^{\frac{3}{2}}}
\end{equation*}

\n and consequently

\begin{equation*}
b_{1,1} = (\kappa^2 - k^2) \frac{\Gamma \! \left( \tfrac{3}{2} \right)}{2} \Res_{t=t_0} \frac{\eta h'}{p^{\frac{3}{2}}}. 
\end{equation*}

\n To compute the above residue, we again use the Taylor expansions in \eqref{eta Taylor} and \eqref{p p ' Taylor}, except that we now need higher-order terms in these expansions. Set again $\tau=t-t_0$. Then, using

\begin{equation*}
\eta(t) = \frac{h_1}{2} \tau + \frac{h_2}{4} \tau^2 + O(\tau^3), 
\qquad \qquad
h'(t) = h_1 + h_2 \tau + O(\tau^2)
\end{equation*}

\n we obtain

\begin{equation*}
\eta(t)h'(t) = \frac{h_1^2}{2} \tau  + \frac{3}{4} h_1h_2 \tau^2 + O(\tau^3).
\end{equation*}

\n Next, recalling that $p(t)=\tau^2\tilde{p}(t)$ and setting $\tilde{p}_0=\tilde{p}(t_0)$ and $\tilde{p}_1=\tilde{p}'(t_0)$, we find

\begin{equation*}
\frac{1}{p^\frac{3}{2}(t)} = \frac{1}{\tau^3 \tilde{p}_0^\frac{3}{2}} \left[ 1 - \frac{3 \tilde{p}_1}{2 \tilde{p}_0} \tau + O(\tau^2)  \right].
\end{equation*}

\n Combining the last two equations, we obtain

\begin{equation*}
\frac{\eta(t)h'(t)}{p^\frac{3}{2}(t)} = \frac{1}{\tau^3 \tilde{p}_0^\frac{3}{2}} \left[ \frac{h_1^2}{2} \tau + \frac{3 h_1}{4}  \left( h_2 - \frac{h_1 \tilde{p}_1}{\tilde{p}_0} \right) \tau^2  +  O(\tau^3)  \right].
\end{equation*}

\n Now, since $p(t) = g(t_0) - g(t) = \tau^2 \tilde{p}(t)$ (see \eqref{p and p tilde}) we can express $3 \tilde{p}_1 / \tilde{p}_0 = g_3/g_2$, where  $g_n = g^{(n)}(t_0)$. This yields the desired residue and consequently

\begin{equation*}
b_{1,1} = (\kappa^2 - k^2) \frac{h_1 \Gamma \! \left( \tfrac{3}{2} \right)}{8 \tilde{p}_0^\frac{3}{2}} \left( 3 h_2 - \frac{h_1 g_3}{g_2} \right)
\end{equation*}

\n Finally, dividing the equation $\CL_1v(\bm{x}_0)=0$ by the nonzero constant (see \eqref{c_N})

\begin{equation*}
c_1 = \left( \kappa^2-k^2 \right)  \frac{h_1 \Gamma \! \left( \tfrac{3}{2} \right)}{16 \tilde{p}_0^{\frac{3}{2}}}
\end{equation*}

\n we obtain an equivalent equation, which is precisely the reduced equation $L_1v(\bm{x}_0) = 0$, i.e., the second equation in \eqref{1st and end cond v}.

\section*{Acknowledgements}
{The author would like to thank Fioralba Cakoni for fruitful discussions about nonscattering.}

\appendix

\section*{Appendix} \label{SECT appendix A}
\setcounter{subsection}{0}
\renewcommand{\thesubsection}{A.\arabic{subsection}}

\setcounter{equation}{0}
\renewcommand{\theequation}{A.\arabic{equation}}

The goal of this section is to prove Theorem~\ref{THM 1}, namely, to obtain a full asymptotic expansion of the integral $I(\lambda)$ as $\lambda\to+\infty$. In the first subsection below we show that the integral $I$, originally defined by \eqref{I cali}, can be reduced to the form \eqref{I cali G g}. In the second subsection, we analyze the reduced form and obtain the desired asymptotic expansion. 

\subsection{Rewriting the integral $I$}

Let us begin by recalling that $\Gamma$ is a simple closed curve parametrized counterclockwise by $\bm{x}(t)=(x_1(t),x_2(t))$ for $t\in[-\pi,\pi]$, and that $g=x_1+ix_2$ and $h=x_1-ix_2$. Further, $\bm{\zeta}(\lambda)=(-i\lambda,\sqrt{\lambda^2+\kappa^2})$ and $\rho(\epsilon)$ is defined by \eqref{rho}. Observe that setting $\epsilon=\lambda^{-1}$ in the definition of $\rho(\epsilon)$ yields

\begin{equation*}
\rho:= \rho(\lambda^{-1})=\sqrt{\lambda^2+\kappa^2}-\lambda.
\end{equation*}

\n Note that we are slightly abusing notation by continuing to denote the above quantity by $\rho$, as this will simplify the notation. Suppressing the $t$-dependence from the notation, we may write

\begin{equation*}
i\bm{x} \cdot\bm{\zeta}(\lambda) = \lambda x_1+i\sqrt{\lambda^2+\kappa^2} x_2 =\lambda\left(x_1+ix_2\right)+i\rho x_2 = g+i\rho x_2.
\end{equation*}

\n Let us set $' = \frac{d}{dt}$. Using the formula

\begin{equation*}
\bm n ds = \bigl(x_2',-x_1'\bigr) dt,
\end{equation*}

\n we can write the integral $I$ as follows:

\begin{equation} \label{I 1}
I(\lambda)=\int_{-\pi}^{\pi}\left\{x_2'\partial_{x_1}v - x_1' \partial_{x_2}v 
+ \left[i(\lambda+\rho)x_1' - \lambda x_2' \right] v\right\} e^{\lambda g+i\rho x_2}dt,
\end{equation}

\n where the function $v$ and all of its partial derivatives are evaluated at $\bm{x}(t)$, i.e., $v=v(\bm{x}(t))$, and similarly for the partial derivatives. To simplify the notation, in what follows, we will suppress the argument of $v$ and its partial derivatives, which will always be understood to be $\bm{x}(t)$. The chain rule implies that

\begin{equation*}
(v)' = x_1' \partial_{x_1}v + x_2' \partial_{x_2}v.
\end{equation*}

\n Consequently, as $A = \partial_{x_1} - i \partial_{x_2}$ we may write

\begin{equation}
x_2'\partial_{x_1}v - x_1' \partial_{x_2}v = -i g'Av+i(v)'.
\end{equation}

\n We now substitute the last expression into \eqref{I 1} and integrate the derivative term $(v)'$ by parts. Since $\Gamma$ is a closed curve there are no boundary terms and we obtain

\begin{equation*}
\begin{split}
I(\lambda)=&\int_{-\pi}^{\pi}\left\{-i g'Av 
+ \left[i(\lambda+\rho)x_1' - \lambda x_2' - i \left( \lambda g' + i \rho x_2' \right) \right] v\right\} e^{\lambda g+i\rho x_2}dt = 
\\
=& \int_{-\pi}^{\pi}\left\{-i g'Av 
+ i \rho h' v\right\} e^{\lambda g+i\rho x_2}dt,
\end{split}
\end{equation*}

\n where the last expression was obtained by simply combining the like terms inside the square brackets. We integrate the first term by parts once more. Note that

\begin{equation*}
\int_{-\pi}^{\pi} g'Av e^{\lambda g+i\rho x_2} dt = \frac{1}{\lambda} \int_{-\pi}^{\pi}  e^{i\rho x_2} Av \left( e^{\lambda g} \right)' dt = - \frac{1}{\lambda} \int_{-\pi}^{\pi} \left( e^{i\rho x_2} Av \right)' e^{\lambda g} dt.
\end{equation*}

\n But then

\begin{equation*}
I(\lambda)= \frac{i}{\lambda} \int_{-\pi}^{\pi} \left[  \left(  e^{i\rho x_2} Av \right)' + \lambda \rho h' e^{i\rho x_2}  \right]  e^{\lambda g} dt.
\end{equation*}

\n It remains to use the equation $(\lambda + \rho)^2 = \lambda^2 + \kappa^2$, which implies that $\lambda \rho = \frac{\kappa^2 - \rho^2}{2}$. Substituting this identity into the above formula completes the proof and shows that $I$ can be rewritten in the form \eqref{I cali G g}.

\subsection{Proof of Theorem~\ref{THM 1}}

Our starting point is the formula \eqref{I cali G g} for the integral $I$, which allows us to write $I(\lambda)=F(\lambda,\lambda^{-1})$, where

\begin{equation} \label{F}
F(\lambda, \epsilon) = \frac{i}{\lambda} \int_{-\pi}^\pi G(t, \epsilon) e^{\lambda g(t) }dt = \frac{i}{\lambda} \int_{\PC} G(t, \epsilon) e^{\lambda g(t) }dt.
\end{equation}

\n Here, in the second equation, we used the analyticity of the integrand to deform the contour of integration from $[-\pi,\pi]$ to $\PC$. Recall that $\PC$ is the contour joining $-\pi$ to $\pi$ and passing through the critical point $t_0$. Along this contour $\Re g$ attains its unique maximum at $t_0$ (see hypothesis (H3) of Section~\ref{SECT main results}).

To obtain the desired asymptotic expansion of $I(\lambda)$ we first obtain the asymptotic expansion of $F(\lambda,\epsilon)$ as $\lambda\to+\infty$ for fixed $\epsilon$, with the expansion being uniform in $\epsilon$. We then set $\epsilon=\lambda^{-1}$, so that the $\epsilon$-dependent coefficients in the expansion become functions of $\lambda^{-1}$. Expanding these coefficients in inverse powers of $\lambda$ and then rearranging and combining like terms, yields the desired asymptotic expansion of $I(\lambda)$. The key property of the function $G(t,\epsilon)$ in this regard is that it is analytic in the two complex variables $(t,\epsilon)\in\CC^2$ in the domain

\begin{equation*}
(t, \epsilon) \in \CT \times \mathbb{D}_{2\delta}, \qquad \qquad
\mathbb{D}_{2\delta} = \{\epsilon \in \CC: |\epsilon| < 2\delta\}.
\end{equation*}

\n Here $\delta>0$ is a small number (we choose the radius to be $2\delta$ for convenience) and recall that $\CT$ is defined in hypothesis (H1). Indeed, this follows from the form of the function $G(t,\epsilon)$ given by \eqref{G}, the fact that the $t$-dependent functions in its definition are all analytic in $\CT$ and the fact that $\rho(\epsilon)$, given by \eqref{rho}, is analytic in $\mathbb{D}_{2\delta}$. The latter holds because we may equivalently rewrite

\begin{equation*}
\rho(\epsilon) = \frac{\kappa^2 \epsilon}{\sqrt{1+\kappa^2 \epsilon^2} + 1},
\end{equation*}

\n which shows that $\rho$ is analytic in a sufficiently small disk $\mathbb{D}_{2\delta}$ (we choose the branch of the square root satisfying $\sqrt{1}=1$).

Now for each fixed $\epsilon\in\mathbb{D}_\delta$, \eqref{F} is an integral to which the method of steepest descent applies directly and yields the full asymptotic expansion. Specifically, \S2.4(iv) of \cite{DLMF} implies that for any $N_0=1,2,...$, as $\lambda\to+\infty$,

\begin{equation} \label{F asymp}
F(\lambda, \epsilon) = \frac{2i}{\lambda^\frac{3}{2}} e^{\lambda g(t_0)} \left[ \sum_{n=0}^{N_0-1} \Gamma \! \left(n + \tfrac{1}{2}\right)
\frac{b_{2n}(\epsilon)}{\lambda^{n}} + O(\lambda^{-N_0}) \right],
\end{equation}

\n where $\Gamma$ denotes the Gamma function and the coefficients of the expansion admit the following residue representation:

\begin{equation} \label{b_2n Res}
b_{2n}(\epsilon) = \frac12 \Res_{t=t_0} \frac{G(t, \epsilon)}{\left[g(t_0)-g(t) \right]^{n + \frac{1}{2}}}.
\end{equation}

\n The key property is that the expansion \eqref{F asymp} is uniform in $\epsilon\in\mathbb{D}_\delta$, i.e., the big-$O$ term in square brackets does not depend on $\epsilon$, but rather on $\delta$. To see this, first let $K\subset\CT$ be a compact set containing the contour $\PC$. Note that the contour $\PC$ and the phase function $g$ do not depend on $\epsilon$. Furthermore, by the analyticity of $G$ and Cauchy's integral formula for derivatives we obtain that for any $j=0,1,...$, the partial derivative $\partial_t^jG(t,\epsilon)$ is uniformly bounded in $\epsilon$. More precisely, there exists $M_j>0$ such that

\begin{equation*}
\sup_{\substack{t \in K \\ |\epsilon| \leq \delta} } |\partial_t^jG(t, \epsilon)| \leq M_j < \infty.
\end{equation*}

\n This provides the desired uniformity in $\epsilon$ of the expansion \eqref{F asymp}. Now, $b_{2n}(\epsilon)$ is analytic in a neighborhood of $\epsilon=0$ (this follows directly from the integral representation of the residue). Hence, we may expand it in a Taylor series:

\begin{equation*}
b_{2n}(\epsilon) = \sum_{m=0}^\infty b_{2n, m} \epsilon^m, \qquad \qquad 
b_{2n, m} = \frac{1}{m!} b_{2n}^{(m)}(0).
\end{equation*}

\n To obtain the expansion of $I(\lambda)=F(\lambda,\lambda^{-1})$, we take $\lambda$ sufficiently large so that $\lambda^{-1}<\delta$ and \eqref{F asymp} holds. We then use the above Taylor expansion up to order $N_0-1-n$ so that, after dividing both sides by $\lambda^n$,

\begin{equation*}
\frac{b_{2n}(\lambda^{-1})}{\lambda^n} = \sum_{m=0}^{N_0-1-n} \frac{b_{2n, m}}{\lambda^{m+n}} + O(\lambda^{-N_0}).
\end{equation*}

\n Substituting the latter into \eqref{F asymp} gives

\begin{equation*}
\begin{split}
F(\lambda, \lambda^{-1}) =& \frac{2i}{\lambda^\frac{3}{2}} e^{\lambda g(t_0)} \left[ \sum_{n=0}^{N_0-1} \Gamma \! \left(n + \tfrac{1}{2}\right)
\sum_{m=0}^{N_0-1-n} \frac{b_{2n, m}}{\lambda^{m+n}} + O(\lambda^{-N_0}) \right] =
\\[.1in]
=& \frac{2i}{\lambda^\frac{3}{2}} e^{\lambda g(t_0)} \left[ \sum_{N=0}^{N_0-1} \frac{1}{\lambda^N} \sum_{n=0}^{N} \Gamma \! \left(n + \tfrac{1}{2}\right) b_{2n, N-n} + O(\lambda^{-N_0}) \right].
\end{split}
\end{equation*}

\n The above formula holds for arbitrary $N_0$, therefore we obtain the full expansion

\begin{equation*}
I(\lambda) \sim \frac{i}{\lambda^\frac{3}{2}} e^{\lambda g(t_0)} \sum_{N=0}^{\infty} \frac{\CR_N}{\lambda^N}, \qquad \qquad
\CR_N = 2\sum_{n=0}^{N} \Gamma \! \left(n + \tfrac{1}{2}\right) b_{2n, N-n}.
\end{equation*}

\n To conclude the proof, it remains to note that this is precisely the formula \eqref{R_n coeff} stated in Theorem~\ref{THM 1}. Indeed, $b_{2n,m}$ is defined as the $m$-th derivative of $b_{2n}(\epsilon)$, and we can move this $\epsilon$-derivative onto $G(t,\epsilon)$ inside the residue in \eqref{b_2n Res}. This follows by representing the residue as an integral and passing the $\epsilon$-derivative under the integral sign, using the uniformity in $\epsilon$.

\bibliographystyle{abbrv}
\bibliography{ref}

\end{document}